\documentclass[12pt,reqno]{amsart}
\usepackage{amsmath,amssymb,amsthm,mathrsfs}
\usepackage{graphicx, cite, times}
\usepackage{cases}
\usepackage{color}
\usepackage{appendix}
\usepackage{enumerate}
\usepackage{bm}
\usepackage{mathtools}
\usepackage[ruled]{algorithm2e}
\usepackage{soul}
\usepackage{verbatim}
\usepackage{subfigure} 
\usepackage{algorithmic}
\usepackage{soul, xcolor}
\usepackage{float}
\usepackage[colorlinks, linkcolor=blue, citecolor=blue]{hyperref}
\usepackage{cite}
\graphicspath{{figures/}}
\usepackage{epstopdf}
\usepackage{rotating}
\usepackage{enumitem} 
\allowdisplaybreaks[4] 
\newtheorem{theo}{Theorem}[section]

\newtheorem{lemm}[theo]{Lemma}

\newtheorem{rema}[theo]{Remark}
\newtheorem{defi}[theo]{Definition}

\newtheorem{assu}[theo]{Assumption}
\numberwithin{equation}{section}
\newcommand{\sgn}{\operatorname{sgn}}
\setlist[enumerate,1]{label=\qquad Case \arabic*:}

\title[Exceptional Points in Time-Modulated Elastic Media ]{Construction of Exceptional Points in Time-Modulated High-Contrast Elastic Media}

\author{Yixian Gao}
\address{School of
Mathematics and Statistics, Center for Mathematics and
Interdisciplinary Sciences, Northeast Normal University, Changchun, Jilin 130024, China. }
\email{gaoyx643@nenu.edu.cn}

\author{Shuguan Ji}
\address{School of
Mathematics and Statistics, Center for Mathematics and
Interdisciplinary Sciences, Northeast Normal University, Changchun, Jilin 130024, China. }
\email{jisg100@nenu.edu.cn}

\author{Shangling  Song$^{*}$}
\address{School of
Mathematics and Statistics, Center for Mathematics and
Interdisciplinary Sciences, Northeast Normal University, Changchun, Jilin 130024, China }
\email{songsl458@nenu.edu.cn}
\thanks{$^{*}$Corresponding author: Shangling Song. 
Email: \texttt{songsl458@nenu.edu.cn}}

\thanks{ The research of YG was supported by NSFC grant
(project number, 12371187) and Science and Technology Development Plan Project of Jilin Province (project number, 20240101006JJ).
The research of SJ was supported by NSFC grants (project number,  12225103, 12071065) }
 


\begin{document}
\maketitle
\begin{abstract}
Spatially periodic elastic metamaterials, comprising hard inclusions within a soft matrix in $d$-dimensional space ($d\geq 2$), exhibit a rich spectrum of physical phenomena. This paper investigates such a model and presents the following contributions. First, we analyze the system's subwavelength quasi-frequencies under static conditions, establishing their functional dependence on the high-contrast parameter. This analysis enables the determination of the subwavelength quasi-frequency range. Second, for time-modulated structures, we derive a system of ordinary differential equations (ODEs) within the subwavelength regime. We demonstrate that this ODE system accurately captures the quasi-frequency behavior of the original elastic system. Finally, leveraging the derived ODEs and Floquet's theorem, we construct concrete examples of first-order asymptotic exceptional points (EPs) in three dimensions.
\end{abstract}

\section{Introduction} 

Metamaterials have demonstrated remarkable capabilities in manipulating wave propagation, deriving their unique effects from intricate spatial architectures that typically exhibit periodic patterns.
A natural progression of this principle involves implementing spatiotemporally modulated structures, where material properties vary spatially and are dynamically modulated over time.
This spatiotemporal engineering approach introduces novel dimensions to wave control mechanisms.
Many researchers have studied issues related to time-dependent media~\cite{doi:10.1137/18M1216894, MendonÃ§a_2002, 9103963, AMMARI2021110594}.
Metamaterials are micro-structured materials composed of subwavelength resonators as their fundamental building blocks.
These engineered structures exhibit exotic and useful properties, with their subwavelength resonance characteristics enabling wave control at scales far below the operating wavelength~\cite{ammari2018subwavelength, pendry2000negative}. 
Crystals refer to extended structures characterized by the periodic repetition of a microscopic arrangement. 
When engineered at subwavelength scales, such crystals demonstrate unique capabilities in wave physics~\cite{kaina2015, luo2003, lemoult2016soda}.

Kato~\cite{Kato1966} introduced the term \emph{exceptional points}.
For a physical system described by the Hamiltonian $H_0 + \eta H_1$, where $\eta$ is a strength parameter. Both  the spectrum and eigenfunctions are analytic in  $\eta$. 
At specific points in the complex $\eta$-plane, two energy levels coalesce. Crucially, this merging differs from conventional degeneracy: the eigenspace collapses from two dimensions to one as the corresponding eigenvectors coalesce, eliminating any two-dimensional subspace.
When both $H_0$ and $H_1$ are real symmetric matrices, EPs can only emerge for complex $\eta$. 
Consequently, at EPs, the full problem $H_0+\eta H_1$ ceases to be Hermitian; however, in its matrix representation, it retains complex symmetry. 
These scenarios are detailed in prior studies~\cite{bender1978advanced, heiss1990avoided, heiss1991transitional, heiss1994phase}.
If the Hamiltonian is non-Hermitian, its eigenvalues are complex in general.
By tuning the amount of non-Hermiticity, one can make two eigenvalues coalesce at one point, and the corresponding eigenvectors coalesce into one, creating EPs~\cite{heiss2004exceptional}.

EPs exhibit diverse applications.
Numerous observable phenomena provide at least indirect evidence of the physical effects caused by EPs, see~\cite{latinne1995laser, shuvalov2000singular, korsch2003stark}.
One crucial role of these points is to enhance the sensitivity of sensors and improve detection accuracy \cite{MR4206386, MR4512271, Chen:2017bjr, wiersig2014, wiersig2016}.
For a sensor operating at a second-order EPs—where both eigenvalues and their corresponding eigenvectors coalesce—a sufficiently small perturbation $\eta$ induces a complex frequency splitting $\Delta\omega$ (i.e., energy-level splitting) that scales as $\Delta\omega \propto \eta^{1/2}$~\cite{Chen:2017bjr}.
EPs were discovered in non-Hermitian systems possessing parity-time ($\mathcal{PT}$) symmetry, and the characteristic of EPs is the transition of the spectrum from purely real to complex~\cite{ruter2010observation, feng2014single}. 
In high-contrast subwavelength resonators, $\mathcal{PT}$ symmetric structures are considered to engineer EPs~\cite{MR4206386, MR4512271}.
In~\cite{dominguez2020}, researchers investigate an aluminum double-torsion-pillar oscillator, thereby revealing the phenomenon of EPs in the field of elastodynamics.
Regarding EPs phenomena in continuous elastic media, the study~\cite{rosa2021exceptional} describes the essential dynamic characteristics of $\mathcal{PT}$ symmetric elastic media exhibiting EPs, including their sensitivity to perturbations, which goes beyond the linear dependency commonly encountered in Hermitian systems.

High-contrast resonators in subwavelength metamaterials enable diverse wave phenomena~\cite{doi:10.1137/21M1449427, MR4563238, MR4206386, MR4512271, ammari2018subwavelength, AMMARI2021110594, MR3906861, ammari2020honeycomb}, with periodic elastic crystals being extensively studied both theoretically and experimentally. We investigate a high-contrast elastic system in $d$-dimensional space comprising hard elastic inclusions embedded in a soft elastic matrix, where time modulation is confined to resonator interiors. Departing from prior studies~\cite{chen2025, ren2025, li2024resonant}, we generalize the density to a diagonal matrix representation, enabling modeling of anisotropic media and significantly broadening the framework's applicability.

While EPs in elastic systems have been previously examined, this work establishes a systematic methodology for constructing asymptotic EPs in periodic elastic structures. Compared to acoustic systems, we demonstrate that constructing asymptotic EPs is more complex in elastic systems--with complexity governed not only by resonator geometry but also by the modulation scheme. Through analysis of representative geometries, we elucidate how structural design fundamentally determines EPs formation.
In the small-parameter limit $\eta\to 0$, asymptotic EPs for ODE systems can be constructed via asymptotic Floquet theory~\cite{MR4389752}. Although $n$-th-order asymptotic EPs are not genuine EPs in the strict sense, they hold substantial theoretical value and practical potential for non-Hermitian spectral analysis and functional device design.
In this work, while we consider first-order perturbations of the coefficient matrix, the perturbation order of the Floquet exponent matrix is higher. We thus focus on explicitly constructing first-order asymptotic EPs. This foundation facilitates the development of higher-order asymptotic EPs and provides a pathway to genuine EPs, while also outlining future research directions.

Given the periodic geometry, 
we apply the Floquet transform to the elastic wave equation, 
yielding the system described in~\eqref{f5}. 
We first analyze the static case (non-modulated) where the Lam\'e parameters and density remain constant both inside and outside the resonators. 
For time-harmonic elastic waves,
 this yields the asymptotic relationship 
\begin{align*}
	\omega=\mathcal{O}(\varepsilon^{1/2})
\end{align*}
between the quasi-frequency $\omega$ and density contrast parameter $\varepsilon$.
This establishes the theoretical foundation for identifying subwavelength frequency regimes in the modulated case.
We next consider resonator-interior time modulation of density, generating a coupled Lam\'e system~\eqref{f20} in the frequency domain. 
Within the subwavelength frequency regime, we derive an ODE system~ \eqref{f26} that predominantly governs the quasi-frequency spectrum of the original elastic system~\eqref{f5}. 
The corresponding  ODE system takes the form
\begin{align}\label{f47}
	\frac{{\rm d}\bm{x}(t)}{{\rm d}t} = \bm{A}(c,t)\bm{x}(t),
\end{align}
where $\bm{A}(c,\cdot): \mathbb{R} \to \mathbb{C}^{n \times n}$ is a $T$-periodic matrix-valued function.
A parameter value $c_0$ constitutes an EP for~\eqref{f47} when the eigenvalues and corresponding eigenvectors of the fundamental solution matrix $\bm{X}(c_0, T)$ simultaneously coalesce. 
To characterize EPs,   consider the  one-parameter system with $\bm{A}(t) \in \mathbb{C}^{n\times n}$ being $T$-periodic. Let $\bm{X}(t)$ be the fundamental matrix solution of
\begin{align*}
	\frac{{\rm d}\bm{x}(t)}{{\rm d}t} = \bm{A}(t)\bm{x}(t).
\end{align*}
By Floquet's theorem, there exists a constant matrix $\bm{F}$ and a $T$-periodic matrix function $\bm{P}(t)$ such that 
\begin{align*}
	\bm{X}(t) = \bm{P}(t)\exp\left(\bm{F} t\right).
\end{align*}
For each eigenvalue ${\rm i}\omega $ of $\bm{F}$, there is a Bloch solution $\bm{x}(t)$ which is $\omega$-quasiperiodic.
At an EPs $c_0$, Floquet's theorem yields the fundamental solution $\bm{X}(c_0, T) = 	\exp(\bm{F}(c_0)T)$, where the Floquet exponent matrix $\bm{F}(c_0)$ becomes non-diagonalizable – indicating coalescence of its eigenvalues and eigenvectors.
Consequently, constructing EPs reduces to analyzing the non-diagonalizability of the Floquet exponent matrix. The detailed derivation of this analysis is presented in Section~\ref{sec5}.
Finally, building upon the derived ODEs system~\eqref{f26} and the asymptotic EPs construction framework, we present several illustrative examples of EPs realization in three-dimensional elastic systems. 
These examples demonstrate the feasibility and adaptability of this approach in complex elastic media, highlighting the interplay between geometric design and modulation schemes.

The paper is structured as follows.
Section~\ref{sec2} formulates the research problem, establishes definitions and notation, and introduces foundational lemmas.
Section~\ref{sec3} analyzes the static case, employing layer potential theory to derive the asymptotic relationship between quasi-frequency $\omega$ and density contrast parameter $\varepsilon$.
Section~\ref{sec4} extends this framework to time-modulated systems with interior density variations, culminating in Theorem~\ref{th5} generalizing the static case results.
Section~\ref{sec5} develops a framework for generating asymptotic EPs in three-dimensional elastic systems with a single resonator.
Appendix~\ref{app1} establishes the inapplicability of our EPs construction methodology to two-dimensional single-resonator configurations.

\section{Problem Formulation}\label{sec2}
In this section, we present the research problem, introduce Floquet-Bloch theory governing periodic differential systems, and review the theoretical foundation of layer potential operators central to this work.

\subsection{Resonator Structure and Elastic Wave Equation}
We consider a periodic geometry generated by linearly independent vectors $\bm{l}_1, \dots, \bm{l}_d$ ($d \geq 2$). The lattice $\Lambda$ and unit cell $Y$ are given by:
\begin{align*}
    \Lambda &= \left\{ \sum_{i=1}^d k_i\bm{l}_i \mid  k_i \in \mathbb{Z},\ i = 1, \cdots, d \right\}, \\
    Y &= \left\{ \sum_{i=1}^d s_i \bm{l}_i \mid s_i \in [0, 1],\ i = 1, \cdots, d \right\}.
\end{align*}
The dual lattice $\Lambda^*$ is generated by basis  vectors $\bm{\gamma}_1, \cdots, \bm{\gamma}_d$ satisfying $\bm{l}_i \cdot \bm{\gamma}_j = 2\pi\delta_{ij}$ ($i,j = 1, \cdots, d$), and is defied as 
\begin{align*}
\Lambda^* = \left\{ \sum_{i=1}^d k_i \bm{\gamma}_i \mid k_i \in \mathbb{Z},\ i = 1, \cdots, d \right\}.
\end{align*}
The (space-)Brillouin zone $Y^*$ is defined as  the torus $Y^*=\mathbb{R}^{d}/\Lambda^*$.
The unit cell  contains  $N$ mutually disjoint connected domains $D_{i}$ ($i=1,2,\cdots, N$) with a connected $C^2$- boundary. The system of resonators $D \subset Y$ and the periodic crystal of resonators $\mathcal{C}$ are then  given by 
\begin{align*}
D=\bigcup_{i=1,\cdots,N}D_{i},
\quad
 \mathcal{C}=\bigcup_{\bm{m}\in\Lambda}D+\bm{m}.
\end{align*}

For the displacement field $\bm{u} = (u_1, u_2, \dots, u_d)^\top$, the Lam\'e  operator is defined as
\begin{align*}
\mathfrak{L}^{\widehat{\lambda},\widehat{\mu}}\bm{u}&=\widehat{\mu}\Delta \bm{u}+(\widehat{\lambda}+\widehat{\mu})\nabla\nabla\cdot \bm{u},
\end{align*}
with the conormal derivative given by
\begin{align*}
\frac{\partial \bm{u}}{\partial\bm{\nu}}&=\widehat{\lambda}(\nabla\cdot \bm{u})\bm{\nu}+\widehat{\mu}(\nabla \bm{u}+\nabla \bm{u}^\top)\bm{\nu},
\end{align*}
where $\bm{\nu}$ denotes  the outward unit normal to the boundary $\partial \mathcal{C}$.

We consider the elastic wave equation:
\begin{align}\label{f45}
\bm {\rho}(\bm{x},t)\partial^{2}_{t}\bm{u}(\bm{x},t)-\mathfrak{L}^{\widehat{\lambda},\widehat{\mu}}\bm{u}(\bm{x},t)=0, \quad \bm{x}\in \mathbb{R}^{d}, t\in \mathbb{R},
\end{align}
where the density matrix $\bm{\rho}(\bm{x},t)=\text{diag}(\rho_1(\bm{x},t), \rho_2(\bm{x},t), \cdots,\rho_d(\bm{x},t))$.
In Section~\ref{sec3}, we first analyze the non-modulated case, assuming the Lam\'{e} parameters $\widehat{\lambda}(\bm{x},t)$, $\widehat{\mu}(\bm{x},t)$, and density matrix $\bm{\rho}(\bm{x},t)$ are piecewise constant in space and time
\begin{align}\label{f46}
\bm{\rho}(\bm{x},t)=
\begin{cases}
\bm{\rho},\quad \bm{x}\in \mathbb{R}^{d}\backslash\overline{\mathcal{C}},\\
\widetilde{\bm{\rho}},\quad \bm{x}\in\mathcal{C},
\end{cases}
\quad
(\widehat{\lambda}(\bm{x},t),\widehat{\mu}(\bm{x},t))=
\begin{cases}
(\lambda,\mu),\quad \bm{x}\in \mathbb{R}^{d}\backslash \overline{\mathcal{C}},\\
(\widetilde{\lambda}, \widetilde{\mu}),\quad \bm{x}\in\mathcal{C},
\end{cases}
\end{align}
where
$\bm{\rho}=\operatorname{diag}(\rho_1,\rho_2,\cdots,\rho_d)$ and $\widetilde{\bm{\rho}}=\operatorname{diag}(\widetilde{\rho}_1,\widetilde{\rho}_2,\cdots,\widetilde{\rho}_d)$.
The background Lam\'{e} parameters satisfy
\begin{align*}
\mu>0,\quad d\lambda+2\mu>0.
\end{align*}

The Lam\'e parameters and mass densities for rigid versus flexible elastic materials are compared through the scaling relations
\begin{align}\label{f7-}
(\widetilde{\lambda}, \widetilde{\mu})=\frac{1}{\delta}(\lambda,\mu),
\quad 
\widetilde{\bm{\rho}}=\frac{1}{\varepsilon}\bm{\rho}.
\end{align}
The contrast parameters $\delta$ and $\varepsilon$ satisfy the asymptotic conditions
\begin{align}\label{relation}
	\delta\rightarrow0,\quad \varepsilon\rightarrow0,\quad \frac{\delta}{\varepsilon}\leqslant\mathcal{O}(1).
\end{align}

\subsection{Floquet–Bloch Theory}

Floquet theory provides a powerful framework for analyzing differential equations with periodic coefficients in time.
Consider the  ODE system
\begin{align}\label{eq:floquet_system}
\begin{cases}
\displaystyle
\frac{{\rm d}\bm{x}(t)}{{\rm d}t} = \bm{A}(t) \bm{x}(t), \\[2mm]
\bm{x}(0) = \bm{x}_0,
\end{cases}
\end{align}
where $\bm{A}(t)$ is a $T$-periodic $n \times n$ complex matrix-valued function.
The theoretical foundations and applications discussed below are detailed in~\cite{doi:10.1137/21M1449427, MR4563238}.

\begin{theo}[Floquet's theorem]
Let $\bm{X}(t)$ be the fundamental matrix solution of \eqref{eq:floquet_system} satisfying the initial condition $\bm{X}(0) = \bm{I}_n$, 
where $\bm{I}_n$ is the $n \times n$ identity matrix. 
There exist a constant matrix $\bm{F}$ and a $T$-periodic matrix function $\bm{P}(t)$ such that
\begin{align*}
	\bm{X}(t) = \bm{P}(t) \exp(\bm{F}t).
\end{align*}
\end{theo}

In particular, $\bm{P}(0) = \bm{I}_n$, and hence the monodromy matrix satisfies $\bm{X}(T) = \exp(\bm{F}T)$.
Floquet's theorem implies that for each eigenvalue $\mathrm{i}\omega$ of $\bm{F}$, 
there exists a Bloch solution $\bm{x}(t)$ that is $\omega$-quasiperiodic, i.e.,
\begin{align*}
	\bm{x}(t + T) = e^{\mathrm{i}\omega T} \bm{x}(t).
\end{align*}
Since $\omega$ is defined modulo the modulation frequency $\Omega = 2\pi/T$,  we define the (time-)Brillouin zone $Y^*_t$ as the quotient space
\begin{align*}
	Y^*_t = \mathbb{C} / (\Omega \mathbb{Z}).
\end{align*}

\begin{defi}
The value $e^{\mathrm{i}\omega}$ is called a {characteristic multiplier}. 
We refer to $\omega$ as the {quasi-frequency} and to $\mathrm{i}\omega$ as the {Floquet exponent}.
\end{defi}

When the matrix $\bm{A}$ is time-independent, the solution to system~\eqref{eq:floquet_system} takes the form
\begin{align*}
\bm{x}(t) = \exp(\bm{A} t)  \bm{x}(0).
\end{align*}
Consequently, the Floquet exponents coincide with the eigenvalues of $\bm{A}$. 
Since these exponents are defined modulo $\mathrm{i}\Omega$, we introduce the \emph{folding number} to characterize this periodic identification.

\begin{defi}[Folding number]
For a time-independent matrix $\bm{A}$, let $\omega_A$ denote the imaginary part of one of its eigenvalues. 
Then $\omega_A$ admits a unique decomposition
\begin{align*}
\omega_A = \omega_0 + m\Omega,
\end{align*}
where $\omega_0 \in (-\Omega/2, \Omega/2]$ and $m \in \mathbb{Z}$. 
The integer $m$ is called the \emph{folding number}.
\end{defi}

\begin{defi}[Floquet transform]
Given a lattice $\Lambda \subset \mathbb{R}^d$, its dual lattice $\Lambda^*$, 
and the fundamental domain $Y^* = \mathbb{R}^d / \Lambda^*$, 
the Floquet transform of a function $f \in L^2(\mathbb{R}^d)$ is defined as
\begin{align*}
	\mathscr{F}[f](\bm{x}, \bm{\alpha}) 
	= \sum_{\bm{\ell} \in \Lambda} f(\bm{x} - \bm{\ell}) e^{\mathrm{i} \bm{\alpha} \cdot \bm{x}},
	\quad \bm{x} \in \mathbb{R}^d,\ \bm{\alpha} \in Y^*.
\end{align*}
The parameter $\bm{\alpha}$ is called the \emph{quasiperiodicity} or \emph{quasimomentum}.
\end{defi}

The Floquet transform serves as a periodic analogue of the Fourier transform. 
Specifically, for fixed quasimomentum $\bm{\alpha} \in Y^*$, the transformed function $\mathscr{F}[f]$ exhibits:
\begin{itemize}
    \item Periodicity in $\bm{\alpha}$: $\mathscr{F}[f](\bm{x}, \bm{\alpha} + \bm{q}) = \mathscr{F}[f](\bm{x}, \bm{\alpha})$ for all $\bm{q} \in \Lambda^*$,
    \item Quasiperiodicity in $\bm{x}$: $\mathscr{F}[f](\bm{x} + \bm{\ell}, \bm{\alpha}) = e^{\mathrm{i} \bm{\alpha} \cdot \bm{\ell}} \mathscr{F}[f](\bm{x}, \bm{\alpha})$ for all $\bm{\ell} \in \Lambda$.
\end{itemize}
The quasiperiodicity relation
\begin{align*}
\mathscr{F}[f](\bm{x} + \bm{\ell}, \bm{\alpha}) = e^{\mathrm{i} \bm{\alpha} \cdot \bm{\ell}} \mathscr{F}[f](\bm{x}, \bm{\alpha}), \quad \bm{\ell} \in \Lambda,
\end{align*}
known as the \emph{Floquet condition}, implies that the function is completely determined by its restriction to a single unit cell $Y$ of the lattice $\Lambda$.
For any linear partial differential operator $\mathcal{G}(\bm{x}, \partial_{\bm{x}})$ with $\Lambda$-periodic coefficients, the Floquet transform satisfies the commutation relation:
\begin{align*}
\mathscr{F}[\mathcal{G} f](\bm{x}, \bm{\alpha}) 
= \mathcal{G}(\bm{x}, \partial_{\bm{x}}) \mathscr{F}[f](\bm{x}, \bm{\alpha}).
\end{align*}
The following theorem holds (see Theorem~7.1 in \cite{2009Layer}).
\begin{theo}[Plancherel-type theorem]
The Floquet transform
\begin{align*}
\mathscr{F} : L^2(\mathbb{R}^d) \to L^2(Y^*; L^2(Y))
\end{align*}
is an isometry. Its inverse is given by
\begin{align*}
\mathscr{F}^{-1}[g](\bm{x}) = \frac{1}{|Y^*|} \int_{Y^*} g(\bm{x}, \bm{\alpha})  \mathrm{d}\bm{\alpha},
\end{align*}
where $g \in L^2(Y^*; L^2(Y))$ is extended from the unit cell $Y$ to $\bm{x} \in \mathbb{R}^d$ via the Floquet condition (i.e., quasiperiodicity in $\bm{x}$).
\end{theo}

We apply the Floquet transform jointly in space $\bm{x}$ and time $t$ to the elastic wave equation~\eqref{f45}, seeking solutions that are quasiperiodic in $t$. The transformed solution vector is defined as 
\begin{align*}
	\bm{U} = \mathscr{F}_{\!\bm{x},t}[\bm{u}] := \big( \mathscr{F}_{\!\bm{x},t}[u_{1}], \mathscr{F}_{\!\bm{x},t}[u_{2}], \dots, \mathscr{F}_{\!\bm{x},t}[u_{d}] \big)^\top.
\end{align*}
This yields the transformed system
\begin{align}\label{f5}
\begin{cases}
\bm{\rho}(\bm{x},t) \partial_{t}^2 \bm{U}(\bm{x},t) - \mathfrak{L}^{\widehat{\lambda},\widehat{\mu}} \bm{U}(\bm{x},t) = \mathbf{0}, \\[2pt]
\bm{U}(\bm{x} + \bm{\ell}, t) = e^{\mathrm{i}\bm{\alpha} \cdot \bm{\ell}} \bm{U}(\bm{x},t) & \forall \bm{\ell} \in \Lambda, \\[2pt]
\bm{U}(\bm{x}, t + T) = e^{\mathrm{i}\omega T} \bm{U}(\bm{x},t),
\end{cases}
\end{align}
where the spatial quasiperiodicity is characterized by quasimomentum $\bm{\alpha}$, and the temporal quasiperiodicity by quasi-frequency $\omega$.

Since $\bm{U}(\bm{x},t)e^{-\mathrm{i}\omega t}$ is $T$-periodic in $t$, it admits the Fourier series expansion
\begin{align*}
\bm{U}(\bm{x},t) = e^{\mathrm{i}\omega t} \sum_{n=-\infty}^{\infty} \bm{v}_{n}(\bm{x}) e^{\mathrm{i} n \Omega t},
\end{align*}
where $\Omega = 2\pi/T$ is the modulation frequency.
The periodic structure of the (time-)Brillouin zone $Y^*_t$ prevents direct application of conventional subwavelength frequency concepts to quasi-frequencies. 
We specifically consider the scaling $\Omega = \mathcal{O}(\varepsilon^{1/2})$ (as established in Section~\ref{sec4}), with $Y^*_t$ defined at this same scaling. 
By construction, $Y^*_t$ contains infinitely many quasi-frequencies resulting from frequency folding outside the subwavelength regime. 
Following \cite{AMMARI2021110594}, we define the \emph{subwavelength quasi-frequencies} for elastic wave propagation as those satisfying specific scaling conditions relative to $\delta$ and $\varepsilon$.

\begin{defi}[Subwavelength quasi-frequency]\label{defi1}
A quasi-frequency $\omega = \omega(\delta, \varepsilon) \in Y^*_t$ of system~\eqref{f5} is called \emph{subwavelength} if there exists a corresponding nontrivial solution $\bm{U}(\bm{x},t)$, continuous in $\delta$ and $\varepsilon$, with the Fourier expansion
\begin{align*}
\bm{U}(\bm{x},t) = e^{\mathrm{i}\omega t} \sum_{n=-\infty}^{\infty} \bm{v}_{n}(\bm{x}) e^{\mathrm{i} n \Omega t},
\end{align*}
and for some integer-valued function $M = M(\delta,\varepsilon)$, satisfying as $\delta \to 0$, $\varepsilon \to 0$:
\begin{itemize}
    \item The energy concentration condition:
    \begin{align*}
    \sum_{n=-\infty}^{\infty} \|\bm{v}_{n}\|_{L^2(Y)} 
    = \sum_{n=-M}^{M} \|\bm{v}_{n}\|_{L^2(Y)} + o(1),
    \end{align*}
    
    \item The asymptotic limits:
    \begin{align*}
    \omega \to 0 \quad \text{and} \quad M\Omega \to 0,
    \end{align*}
\end{itemize}
where the $L^2$-norm over the unit cell $Y$ is defined as
\begin{align*}
\|\bm{v}_n\|_{L^2(Y)} = \left( \int_{Y} \bm{v}_n(\bm{x}) \cdot \overline{\bm{v}_n(\bm{x})}  \mathrm{d}\bm{x} \right)^{1/2}.
\end{align*}
\end{defi}

\subsection{Layer Potential Operators}

We introduce the $d$-dimensional quasiperiodic fundamental solution matrix 
\begin{align*}
\bm{G}^{\bm{\alpha},\Theta,\omega}(\bm{x},\bm{y}) = \big( G^{\bm{\alpha},\Theta,\omega}_{ij}(\bm{x},\bm{y}) \big)_{i,j=1}^d,
\end{align*}
which satisfies the partial differential equation
\begin{align*}
\left( \mathfrak{L}^{\lambda,\mu} + \omega^2 \Theta \right) \bm{G}^{\bm{\alpha},\Theta,\omega}(\bm{x},\bm{y})
= \sum_{\bm{\ell} \in \Lambda} \delta(\bm{x} - \bm{y} - \bm{\ell}) e^{\mathrm{i}\bm{\ell}\cdot\bm{\alpha}} \bm{I}_d,
\end{align*}
where $\Theta = \operatorname{diag}(\theta_1, \theta_2, \dots, \theta_d)$ and we assume the non-resonance condition
\begin{align*}
\frac{\omega \sqrt{\theta_i}}{\sqrt{\mu}} \neq |\bm{q} + \bm{\alpha}| \quad \text{for all} \quad i = 1, 2, \dots, d, \quad \bm{q} \in \Lambda^*.
\end{align*}
Applying the Poisson summation formula
\begin{equation}
\sum_{\bm{\ell} \in \Lambda} \delta(\bm{x} - \bm{y} - \bm{\ell}) e^{\mathrm{i} \bm{\ell} \cdot \bm{\alpha}} \bm{I}_d
= \frac{1}{|Y|} \sum_{\bm{q} \in \Lambda^*} e^{\mathrm{i} (\bm{q} + \bm{\alpha}) \cdot (\bm{x} - \bm{y})} \bm{I}_d,
\end{equation}
the diagonal entries $(i = 1,\dots,d)$ of the fundamental solution are given by
\begin{align} \label{f32}
G^{\bm{\alpha},\Theta,\omega}_{ii}(\bm{x},\bm{y}) 
&= \frac{1}{|Y|} \sum_{\bm{q} \in \Lambda^*}
\Bigg[
  \Bigg( \frac{1}{\omega^2 \theta_{i} - \mu|\bm{q} + \bm{\alpha}|^2} 
  - (\lambda + \mu) \sum_{\substack{s=1 \\ s \neq i}}^d 
      \frac{(\bm{q} + \bm{\alpha})_s^2}
           {(\omega^2 \theta_{i} - \mu|\bm{q} + \bm{\alpha}|^2)
            (\omega^2 \theta_{s} - \mu|\bm{q} + \bm{\alpha}|^2)}
  \Bigg) \nonumber \\
&\quad \times \left( 1 - (\lambda + \mu) \sum_{s=1}^d 
      \frac{(\bm{q} + \bm{\alpha})_s^2}{\omega^2 \theta_{s} - \mu|\bm{q} + \bm{\alpha}|^2} 
  \right)^{-1} 
\Bigg] e^{\mathrm{i}(\bm{q} + \bm{\alpha}) \cdot (\bm{x} - \bm{y})},
\end{align}
and the off-diagonal entries $(i \neq j)$ by
\begin{align} \label{eq:fundamental_offdiag}
G^{\bm{\alpha},\Theta,\omega}_{ij}(\bm{x},\bm{y}) 
&= \frac{1}{|Y|} (\lambda + \mu) \sum_{\bm{q} \in \Lambda^*}
\Bigg[
  \frac{(\bm{q} + \bm{\alpha})_i (\bm{q} + \bm{\alpha})_j}
       {(\omega^2 \theta_{i} - \mu|\bm{q} + \bm{\alpha}|^2)
        (\omega^2 \theta_{j} - \mu|\bm{q} + \bm{\alpha}|^2)} \nonumber \\
&\quad \times \left( 1 - (\lambda + \mu) \sum_{s=1}^d 
      \frac{(\bm{q} + \bm{\alpha})_s^2}{\omega^2 \theta_{s} - \mu|\bm{q} + \bm{\alpha}|^2} 
  \right)^{-1} 
\Bigg] e^{\mathrm{i}(\bm{q} + \bm{\alpha}) \cdot (\bm{x} - \bm{y})},
\end{align}
where $(\bm{q} + \bm{\alpha})_s$ denotes the $s$-th component of $\bm{q} + \bm{\alpha}$.

For the region $D$ defined previously, we introduce the single-layer potential $\mathcal{S}^{\bm{\alpha},\Theta,\omega}_{D}$
and the Neumann-Poincar\'{e} operator $(\mathcal{K}^{-\bm{\alpha},\Theta,\omega}_{D})^*$ associated with the fundamental solution $\bm{G}^{\bm{\alpha},\Theta,\omega}(\bm{x},\bm{y})$. For a given density $\bm{\varphi} \in H^{-1/2}(\partial D)^d$,
\begin{align}
\mathcal{S}^{\bm{\alpha},\Theta,\omega}_{D}[\bm{\varphi}](\bm{x})
&= \int_{\partial D} \bm{G}^{\bm{\alpha},\Theta,\omega}(\bm{x},\bm{y}) \bm{\varphi}(\bm{y})  \mathrm{d}\sigma(\bm{y}), \label{f2} \\
(\mathcal{K}^{-\bm{\alpha},\Theta,\omega}_{D})^*[\bm{\varphi}](\bm{x})
&= \mathrm{p.v.} \int_{\partial D} \frac{\partial \bm{G}^{\bm{\alpha},\Theta,\omega}(\bm{x},\bm{y})}{\partial \bm{\nu}(\bm{x})} \bm{\varphi}(\bm{y})  \mathrm{d}\sigma(\bm{y}). \nonumber
\end{align}
The potential $\mathcal{S}^{\bm{\alpha},\Theta,\omega}_{D}[\bm{\varphi}]$ satisfies 
\begin{align*}
(\mathfrak{L}^{\lambda,\mu} + \omega^2 \Theta) \bm{u} = \bm{0}.
\end{align*}
Its normal derivative exhibits the jump relation 
\begin{align} \label{f11}
\left. \frac{\partial \mathcal{S}^{\bm{\alpha},\Theta,\omega}_{D}[\bm{\varphi}]}{\partial \bm{\nu}} \right|_{\pm} (\bm{x})
= \left( \pm \frac{1}{2} \mathcal{I} + (\mathcal{K}^{-\bm{\alpha},\Theta,\omega}_{D})^* \right) [\bm{\varphi}](\bm{x}), \quad 
\quad a.e.\,\bm{x}\in \partial D,
\end{align}
where $\mathcal{I}$ denotes the identity operator. 
For further details, see~\cite[Chapter~V]{10.1115/1.3153629}.

Since this study focuses on the low-frequency regime, the fundamental solution admits the following asymptotic expansion.

\begin{lemm}\label{lem6}
For $\bm{\alpha } \neq 0$, and $\omega \to 0$, the fundamental solution $\bm{G}^{\bm{\alpha},\Theta,\omega}(\bm{x},\bm{y})$ admits the asymptotic decomposition
\begin{align*}
G^{\bm{\alpha},\Theta,\omega}_{ij}(\bm{x},\bm{y})
= G^{\bm{\alpha},0,0}_{ij}(\bm{x},\bm{y}) 
+ \sum_{k=1}^{\infty} \omega^{2k} G^{\bm{\alpha},\Theta,0}_{k,ij}(\bm{x},\bm{y}),
\end{align*}
where the components satisfy:

\begin{enumerate}[label=(\arabic*),leftmargin=*]
  \item \textbf{Diagonal components} ($i = j$):
  \begin{align*}
  G^{\bm{\alpha},0,0}_{ii}(\bm{x},\bm{y}) 
  &= -\frac{1}{|Y|} \sum_{\bm{q} \in \Lambda^*} \Bigg( \frac{1}{\mu|\bm{q} + \bm{\alpha}|^2} 
  - \frac{\lambda + \mu}{\mu(\lambda + 2\mu)} \frac{(\bm{q} + \bm{\alpha})_i^2}{|\bm{q} + \bm{\alpha}|^{4}} \Bigg) 
  e^{\mathrm{i}(\bm{q} + \bm{\alpha}) \cdot (\bm{x} - \bm{y})}, \\
  G^{\bm{\alpha},\Theta,0}_{k,ii}(\bm{x},\bm{y}) 
  &= -\frac{1}{|Y|} \sum_{\bm{q} \in \Lambda^*} \Bigg[ \frac{1}{(\lambda + 2\mu)|\bm{q} + \bm{\alpha}|^2} 
  \sum_{\substack{m,n \geq 0 \\ m + n = k}} \Bigg( \frac{\theta_i^{m}}{(\mu|\bm{q} + \bm{\alpha}|^2)^{m}} \\
  &\quad + (\lambda + \mu) \sum_{\substack{m_1,m_2 \geq 0 \\ m_1 + m_2 = m}} 
  \frac{\theta_i^{m_1} \Big( \sum_{\substack{s=1 \\ s \neq i}}^d (\bm{q} + \bm{\alpha})_s^2 \theta_s^{m_2} \Big)}{(\mu|\bm{q} + \bm{\alpha}|^2)^{m+1}} \Bigg) 
  (-1)^{n} T_{n} \Bigg] e^{\mathrm{i}(\bm{q} + \bm{\alpha}) \cdot (\bm{x} - \bm{y})}.
  \end{align*} 

  \item \textbf{Off-diagonal components} ($i \neq j$):
  \begin{align*}
  G^{\bm{\alpha},0,0}_{ij}(\bm{x},\bm{y}) 
  &= \frac{1}{|Y|} \sum_{\bm{q} \in \Lambda^*} \Bigg( \frac{\lambda + \mu}{\mu(\lambda + 2\mu)} 
  \frac{(\bm{q} + \bm{\alpha})_i (\bm{q} + \bm{\alpha})_j}{|\bm{q} + \bm{\alpha}|^{4}} \Bigg) 
  e^{\mathrm{i}(\bm{q} + \bm{\alpha}) \cdot (\bm{x} - \bm{y})}, \\
  G^{\bm{\alpha},\Theta,0}_{k,ij}(\bm{x},\bm{y}) 
  &= \frac{1}{|Y|} \sum_{\bm{q} \in \Lambda^*} \Bigg[ \frac{(\lambda + \mu)(\bm{q} + \bm{\alpha})_i (\bm{q} + \bm{\alpha})_j}{\mu(\lambda + 2\mu)|\bm{q} + \bm{\alpha}|^{4}} \\
  &\quad \times \sum_{\substack{m,n \geq 0 \\ m + n = k}} \Bigg( 
  \frac{ \sum_{\substack{m_1,m_2 \geq 0 \\ m_1 + m_2 = m}} \theta_i^{m_1} \theta_j^{m_2} }{(\mu|\bm{q} + \bm{\alpha}|^2)^{m}} 
  (-1)^{n} T_{n} \Bigg) \Bigg] e^{\mathrm{i}(\bm{q} + \bm{\alpha}) \cdot (\bm{x} - \bm{y})}.
  \end{align*}
\end{enumerate}
Here $T_0 = 1$, and for $n \geq 1$,
\begin{align*}
T_n &= \frac{\mu(\lambda + \mu)}{\lambda + 2\mu} 
\frac{ \sum_{s=1}^d (\bm{q} + \bm{\alpha})_s \theta_s^n }{(\mu|\bm{q} + \bm{\alpha}|^2)^{n+1}} \\
&+ \left( \frac{\mu(\lambda + \mu)}{\lambda + 2\mu} \right)^2 
\sum_{\substack{n_1,n_2 \geq 0 \\ n_1 + n_2 = n}} 
\frac{ \prod_{j=1}^2 \left( \sum_{s=1}^d (\bm{q} + \bm{\alpha})_s \theta_s^{n_j} \right) }{(\mu|\bm{q} + \bm{\alpha}|^2)^{n+1}} \\
&+ \cdots \\
&+ \left( \frac{\mu(\lambda + \mu)}{\lambda + 2\mu} \right)^{n-1} 
\sum_{\substack{n_1,\dots, n_{n-1} \geq 0 \\ n_1 + \cdots + n_{n-1} = n}} 
\frac{ \prod_{j=1}^{n-1} \left( \sum_{s=1}^d (\bm{q} + \bm{\alpha})_s \theta_s^{n_j} \right) }{(\mu|\bm{q} + \bm{\alpha}|^2)^{n+1}} \\
&+ \left( \frac{\mu(\lambda + \mu)}{\lambda + 2\mu} \right)^{n} 
\frac{ \left( \sum_{s=1}^d (\bm{q} + \bm{\alpha})_s \theta_s \right)^n }{(\mu|\bm{q} + \bm{\alpha}|^2)^{n+1}}.
\end{align*}
\end{lemm}

\begin{proof}
We provide the proof for the diagonal case ($i = j$); the off-diagonal case ($i \neq j$) follows analogously. 

As $\omega \to 0$, we have the expansion:
\begin{align}\label{f33}
\frac{1}{\omega^2\theta_i - \mu|\bm{q} + \bm{\alpha}|^2} 
= -\frac{1}{\mu|\bm{q} + \bm{\alpha}|^2} \sum_{k=0}^\infty \left( \frac{\theta_i}{\mu|\bm{q} + \bm{\alpha}|^2} \right)^k \omega^{2k}.
\end{align}
Consequently, 
\begin{align}\label{f34}
\frac{(\bm{q} + \bm{\alpha})_s^2}{(\omega^2 \theta_{i} - \mu|\bm{q} + \bm{\alpha}|^2)(\omega^2 \theta_{s} - \mu|\bm{q} + \bm{\alpha}|^2)} 
= \frac{(\bm{q} + \bm{\alpha})_s^2}{(\mu|\bm{q} + \bm{\alpha}|^2)^2} \sum_{k=0}^\infty \left( \frac{ \sum_{\substack{k_1,k_2 \geq 0 \\ k_1 + k_2 = k}} \theta_i^{k_1} \theta_s^{k_2} }{(\mu|\bm{q} + \bm{\alpha}|^2)^k} \right) \omega^{2k},
\end{align}
and
\begin{equation}\label{f35}
\begin{aligned}
& \left( 1 - (\lambda + \mu) \sum_{s=1}^d \frac{(\bm{q} + \bm{\alpha})_s^2}{\omega^2 \theta_{s} - \mu|\bm{q} + \bm{\alpha}|^2} \right)^{-1} \\
&= \left( 1 + (\lambda + \mu) \sum_{k=0}^\infty \frac{ \sum_{s=1}^d (\bm{q} + \bm{\alpha})_s^2 \theta_s^k }{(\mu|\bm{q} + \bm{\alpha}|^2)^{k+1}} \omega^{2k} \right)^{-1} \\
&= \left( \frac{\lambda + 2\mu}{\mu} + (\lambda + \mu) \sum_{k=1}^\infty \frac{ \sum_{s=1}^d (\bm{q} + \bm{\alpha})_s^2 \theta_s^k }{(\mu|\bm{q} + \bm{\alpha}|^2)^{k+1}} \omega^{2k} \right)^{-1} \\
&= \frac{\mu}{\lambda + 2\mu} \left( 1 + \frac{\mu(\lambda + \mu)}{\lambda + 2\mu} \sum_{k=1}^\infty \frac{ \sum_{s=1}^d (\bm{q} + \bm{\alpha})_s^2 \theta_s^k }{(\mu|\bm{q} + \bm{\alpha}|^2)^{k+1}} \omega^{2k} \right)^{-1} \\
&= \frac{\mu}{\lambda + 2\mu} \sum_{k=0}^\infty (-1)^k T_k \omega^{2k}.
\end{aligned}
\end{equation}
Here $T_0 = 1$, and for $k \geq 1$,
\begin{align*}
T_k ={}& \frac{\mu(\lambda + \mu)}{\lambda + 2\mu} \frac{ \sum_{s=1}^d (\bm{q} + \bm{\alpha})_s^2 \theta_s^k }{(\mu|\bm{q} + \bm{\alpha}|^2)^{k+1}} \\
&+ \left( \frac{\mu(\lambda + \mu)}{\lambda + 2\mu} \right)^2 \sum_{\substack{k_1,k_2 \geq 0 \\ k_1 + k_2 = k}} \frac{ \prod_{j=1}^2 \left( \sum_{s=1}^d (\bm{q} + \bm{\alpha})_s^2 \theta_s^{k_j} \right) }{(\mu|\bm{q} + \bm{\alpha}|^2)^{k+1}} \\
&+ \cdots \\
&+ \left( \frac{\mu(\lambda + \mu)}{\lambda + 2\mu} \right)^k \frac{ \left( \sum_{s=1}^d (\bm{q} + \bm{\alpha})_s^2 \theta_s \right)^k }{(\mu|\bm{q} + \bm{\alpha}|^2)^{k+1}}.
\end{align*}

Combining \eqref{f33}, \eqref{f34}, and \eqref{f35}, we establish the diagonal case of the lemma.
\end{proof}

According to Lemma~\ref{lem6}, for $\bm{\alpha }\neq 0$, and $\omega \to 0$, the operators 
$\mathcal{S}^{\bm{\alpha},\Theta,\omega}_{D}$ and $(\mathcal{K}^{-\bm{\alpha},\Theta,\omega}_{D})^*$ admit the asymptotic expansions:
\begin{equation}\label{f9}
\begin{aligned}
\mathcal{S}^{\bm{\alpha},\Theta,\omega}_{D}
&= \mathcal{S}^{\bm{\alpha},0,0}_{D}
  + \sum_{k=1}^{\infty} \omega^{2k} \mathcal{S}^{\bm{\alpha},\Theta,0}_{D,k}, \\
(\mathcal{K}^{-\bm{\alpha},\Theta,\omega}_{D})^*
&= (\mathcal{K}^{-\bm{\alpha},0,0}_{D})^*
  + \sum_{k=1}^{\infty} \omega^{2k} \mathcal{K}^{\bm{\alpha},\Theta,0}_{D,k}.
\end{aligned}
\end{equation}
Here the operators are defined by
\begin{align*}
\mathcal{S}^{\bm{\alpha},\Theta,0}_{D,k}[\bm{\varphi}](\bm{x})
&= \int_{\partial D} \bm{G}^{\bm{\alpha},\Theta,0}_{k}(\bm{x},\bm{y}) \bm{\varphi}(\bm{y})  \mathrm{d}\sigma(\bm{y}), \\
\mathcal{K}^{\bm{\alpha},\Theta,0}_{D,k}[\bm{\varphi}](\bm{x})
&= \mathrm{p.v.} \int_{\partial D} \frac{\partial \bm{G}^{\bm{\alpha},\Theta,0}_{k}(\bm{x},\bm{y})}{\partial \bm{\nu}(\bm{x})} \bm{\varphi}(\bm{y})  \mathrm{d}\sigma(\bm{y}).
\end{align*}
The series in \eqref{f9} are convergent. Moreover, for each $k \geq 1$, 
$\mathcal{S}^{\bm{\alpha},\Theta,0}_{D,k}$ and $\mathcal{K}^{\bm{\alpha},\Theta,0}_{D,k}$ 
are bounded linear operators satisfying
\begin{align*}
&\sup_{k \geq 1} \left\| \mathcal{S}^{\bm{\alpha},\Theta,0}_{D,k} \right\|_{\mathcal{B}(H^{-1/2}(\partial D)^d, H^{1/2}(\partial D)^d)} < \infty, \\
&\sup_{k \geq 1} \left\| \mathcal{K}^{\bm{\alpha},\Theta,0}_{D,k} \right\|_{\mathcal{B}(H^{-1/2}(\partial D)^d, H^{-1/2}(\partial D)^d)} < \infty,
\end{align*}
where $\mathcal{B}(S_1, S_2)$ denotes the Banach space of bounded linear operators from $S_1$ to $S_2$ equipped with the operator norm.

We then obtain the asymptotic expansions:
\begin{align}
	\mathcal{S}^{\bm{\alpha},\Theta,\omega}_{D} 
	&= \mathcal{S}^{\bm{\alpha},0,0}_{D} + \omega^{2} \mathcal{S}^{\bm{\alpha},\Theta,0}_{D,1} + \mathcal{O}(\omega^{4}), \label{f9-} \\
	(\mathcal{K}^{-\bm{\alpha},\Theta,\omega}_{D})^*
	&= (\mathcal{K}^{-\bm{\alpha},0,0}_{D})^* + \omega^{2} \mathcal{K}^{\bm{\alpha},\Theta,0}_{D,1} + \mathcal{O}(\omega^{4}). \label{f10-}
\end{align}

The following lemma naturally extends results from \cite{ren2025, Lipton2022}.
\begin{lemm}\label{lem2}
The operator $\mathcal{S}^{\bm{\alpha},0,0}_{D} : H^{-1/2}(\partial D)^d \to H^{1/2}(\partial D)^d$ 
is invertible for each nonzero $\bm{\alpha}$.
\end{lemm}

Let \(\bm{e}_s\) (\(s = 1, 2, \dots, d\)) denote the standard unit basis vectors in \(\mathbb{R}^d\), and let \(\chi_{\partial D_i}\) be the characteristic function of the boundary \(\partial D_i\) for \(i = 1, 2, \dots, N\). 
Leveraging the invertibility of \(\mathcal{S}^{\bm{\alpha},0,0}_{D}\), we define the \(d\)-dimensional vector 
\begin{align}\label{f18}
\bm{C}^{\bm{\alpha},s}_{ij} 
:= \int_{\partial D_i} (\mathcal{S}^{\bm{\alpha},0,0}_D)^{-1} [ \bm{e}_{s} \chi_{\partial D_{j}} ] (\bm{x}) \mathrm{d}\sigma(\bm{x}),
\end{align}
for \(s = 1, 2, \dots, d\) and \(i, j = 1, 2, \dots, N\).

The following lemma describes the symmetry properties of \(\bm{C}^{\bm{\alpha},s}_{ij}\).
\begin{lemm}\label{lem7}
Assume that $\bm{\alpha } \neq 0$.
Let \(\bm{C}^{\bm{\alpha},s}_{ij} = (C^{\bm{\alpha},s,1}_{ij}, C^{\bm{\alpha},s,2}_{ij}, \dots, C^{\bm{\alpha},s,d}_{ij})^\top\) be defined as in \eqref{f18}. Then
\begin{align*}
C^{\bm{\alpha},s,s'}_{ij} = \overline{C^{\bm{\alpha},s',s}_{ji}}.
\end{align*}
\end{lemm}
The proof, which requires a key integral formula, is postponed to Section~\ref{sec3}.

\section{Non-modulated case}\label{sec3}
We first consider the non-modulated case, where the Lam\'e constants and mass densities satisfy \eqref{f46} and \eqref{f7-}. 
Applying the elastic wave equation \eqref{f5} to time-harmonic solutions of the form $\bm{U}(\bm{x},t) = e^{\mathrm{i}\omega t} \bm{v}(\bm{x})$ yields:
\begin{align*}
\begin{cases}
(\mathfrak{L}^{\lambda,\mu} + \bm{\rho} \omega^2) \bm{v}(\bm{x}) = \bm{0}, & \bm{x} \in \mathbb{R}^{d} \setminus \overline{\mathcal{C}}, \\[3pt]
(\mathfrak{L}^{\widetilde{\lambda},\widetilde{\mu}} + \widetilde{\bm{\rho}} \omega^2) \bm{v}(\bm{x}) = \bm{0}, & \bm{x} \in \mathcal{C}, \\[3pt]
\bm{v}|_{+} = \bm{v}|_{-}, & \bm{x} \in \partial \mathcal{C}, \\[3pt]
\left. \dfrac{\partial \bm{v}}{\partial \bm{\nu}} \right|_{+} \hspace{-1pt}= \left. \dfrac{\partial \bm{v}}{\partial \bm{\nu}} \right|_{-}, & \bm{x} \in \partial \mathcal{C}, \\[3pt]
\bm{v}(\bm{x}) e^{-\mathrm{i} \bm{\alpha} \cdot \bm{x}} \text{ is periodic is periodic in the whole space}.
\end{cases}
\end{align*}
Using \eqref{f7-}, we obtain the equivalent transmission problem:
\begin{align}\label{f8}
\begin{cases}
(\mathfrak{L}^{\lambda,\mu} + \bm{\rho} \omega^2) \bm{v}(\bm{x}) = \bm{0}, & \bm{x} \in Y \setminus \overline{D}, \\[3pt]
\left(\mathfrak{L}^{\lambda,\mu} + \dfrac{\delta}{\varepsilon} \bm{\rho} \omega^2\right) \bm{v}(\bm{x}) = \bm{0}, & \bm{x} \in D, \\[3pt]
\bm{v}|_{+} = \bm{v}|_{-}, & \bm{x} \in \partial D, \\[3pt]
\delta \left. \dfrac{\partial \bm{v}}{\partial \bm{\nu}} \right|_{+} \hspace{-1pt}= \left. \dfrac{\partial \bm{v}}{\partial \bm{\nu}} \right|_{-}, & \bm{x} \in \partial D, \\[3pt]
\bm{v}(\bm{x}) e^{-\mathrm{i} \bm{\alpha} \cdot \bm{x}} \text{ is periodic is periodic in the whole space}.
\end{cases}
\end{align}

This work establishes an asymptotic expansion for the transmission eigenvalue problem.
\begin{theo}\label{th4}
Assume that $\bm{\alpha } \neq 0$. The frequencies $\omega_i$ satisfying equation \eqref{f8} have the asymptotic approximation 
\begin{align*}
	\omega_i = \sqrt{ -\varepsilon \, \xi_i } + \mathcal{O}(\varepsilon) \quad \text{as} \quad \varepsilon \to 0,
\end{align*}
where $\xi_i$ ($i = 1, \dots, dN$) denote the eigenvalues of the $dN \times dN$ block matrix $\mathbf{H}$, defined by
\begin{align*}
\mathbf{H} =
\begin{pmatrix}
\mathbf{H}_{11} & \cdots & \mathbf{H}_{1N} \\
\vdots & \ddots & \vdots \\
\mathbf{H}_{N1} & \cdots & \mathbf{H}_{NN}
\end{pmatrix}.
\end{align*}
Each $d \times d$ block $\mathbf{H}_{ij}$ is given by
\begin{align*}
\mathbf{H}_{ij} = \frac{1}{|D_i|} \bm{\rho}^{-1} \begin{pmatrix} \mathbf{C}^{\alpha,1}_{ij} & \mathbf{C}^{\alpha,2}_{ij} & \cdots & \mathbf{C}^{\alpha,d}_{ij} \end{pmatrix},
\end{align*}
where $|D_i|$ represents the volume of domain $D_i$.
\end{theo}

To prove the theorem, we first establish the following auxiliary lemma.
\begin{lemm}\label{lem4}
Assume that $\bm{\alpha } \neq 0$. For every $\bm{\phi} \in H^{-1/2}(\partial D)^d$,
\begin{align*}
\int_{\partial D_{j}} \mathcal{K}^{\bm{\alpha},\bm{\rho},0}_{D,1}[\bm{\phi}](\bm{x})  \mathrm{d}\sigma(\bm{x})
= -\bm{\rho} \int_{D_{j}} \mathcal{S}^{\bm{\alpha}, 0,0}_D[\bm{\phi}](\bm{x})  \mathrm{d}\bm{x}.
\end{align*}
\end{lemm}

\begin{proof}
For any $\bm{\phi} \in H^{-\frac{1}{2}}(\partial D)^d$, we first note  that:
\begin{align}\label{f19}
\int_{\partial D_{j}}
& \left( -\frac{1}{2}\mathcal{I} + (\mathcal{K}^{-\bm{\alpha},0, 0}_D)^* \right)[\bm{\phi}](\bm{x})  \mathrm{d}\sigma(\bm{x})\nonumber\\
&= \int_{\partial D_{j}} \frac{\partial}{\partial\bm{\nu}} \mathcal{S}^{\bm{\alpha}, 0,0}_D[\bm{\phi}](\bm{x})  \mathrm{d}\sigma(\bm{x}) \nonumber \\
&= \int_{D_{j}} \mathcal{L}^{\lambda,\mu} \mathcal{S}^{\bm{\alpha},0,0}_D[\bm{\phi}](\bm{x})  \mathrm{d}\bm{x} = 0.
\end{align}

Using asymptotic expansions \eqref{f9-} and \eqref{f10-}, we develop two expressions for the same integral. First:
\begin{align*}
\int_{\partial D_j} 
&\left( -\frac{1}{2}\mathcal{I} + (\mathcal{K}^{-\bm{\alpha},\bm{\rho},\omega}_D)^* \right) [\bm{\phi}](\bm{x})  \mathrm{d}\sigma(\bm{x})\\
&= \int_{\partial D_j} \left( -\frac{1}{2}\mathcal{I} + (\mathcal{K}^{-\bm{\alpha},0,0}_D)^* + \omega^2 \mathcal{K}^{\bm{\alpha},\bm{\rho},0}_{D,1} \right) [\bm{\phi}](\bm{x})  \mathrm{d}\sigma(\bm{x}) + \mathcal{O}(\omega^4) \\
&= \omega^2 \int_{\partial D_j} \mathcal{K}^{\bm{\alpha},\bm{\rho},0}_{D,1}[\bm{\phi}](\bm{x})  \mathrm{d}\sigma(\bm{x}) + \mathcal{O}(\omega^4),
\end{align*}
as $\omega \to 0$, where we used \eqref{f19} to simplify. 

Second, using the properties of layer potentials:
\begin{align*}
\int_{\partial D_j} \left( -\frac{1}{2}\mathcal{I} + (\mathcal{K}^{-\bm{\alpha},\bm{\rho},\omega}_D)^* \right) [\bm{\phi}](\bm{x})  \mathrm{d}\sigma(\bm{x})
&= \int_{\partial D_j} \frac{\partial}{\partial \bm{\nu}} \mathcal{S}^{\bm{\alpha},\bm{\rho},\omega}_D [\bm{\phi}](\bm{x})  \mathrm{d}\sigma(\bm{x}) \\
&= \int_{D_j} \mathcal{L}^{\lambda,\mu} \mathcal{S}^{\bm{\alpha},\bm{\rho},\omega}_D [\bm{\phi}](\bm{x})  \mathrm{d}\bm{x} \\
&= -\bm{\rho} \, \omega^2 \int_{D_j} \mathcal{S}^{\bm{\alpha},\bm{\rho},\omega}_D [\bm{\phi}](\bm{x})  \mathrm{d}\bm{x} \\
&= -\bm{\rho} \, \omega^2 \int_{D_j} \mathcal{S}^{\bm{\alpha},0,0}_D [\bm{\phi}](\bm{x})  \mathrm{d}\bm{x} + \mathcal{O}(\omega^4).
\end{align*}

Equating both expressions at order $\omega^2$  yields:
\begin{align*}
\omega^2 \int_{\partial D_j} \mathcal{K}^{\bm{\alpha},\bm{\rho},0}_{D,1} [\bm{\phi}](\bm{x})  \mathrm{d}\sigma(\bm{x}) + \mathcal{O}(\omega^4)
= -\bm{\rho} \, \omega^2 \int_{D_j} \mathcal{S}^{\bm{\alpha},0,0}_D [\bm{\phi}](\bm{x})  \mathrm{d}\bm{x} + \mathcal{O}(\omega^4).
\end{align*}
Dividing by $\omega^2$ and taking $\omega \to 0$ gives:
\begin{align*}
\int_{\partial D_j} \mathcal{K}^{\bm{\alpha},\bm{\rho},0}_{D,1} [\bm{\phi}](\bm{x})  \mathrm{d}\sigma(\bm{x})
= -\bm{\rho} \int_{D_j} \mathcal{S}^{\bm{\alpha},0,0}_D [\bm{\phi}](\bm{x})  \mathrm{d}\bm{x},
\end{align*}
since the leading-order terms are $\omega$-independent.
\end{proof}

In addition, applying \eqref{f19} yields the following boundary integral relation for all $\bm{\phi} \in H^{-1/2}(\partial D)^d$,
\begin{align}\label{f22}
\int_{\partial D} \left( \dfrac{1}{2} \mathcal{I} + (\mathcal{K}^{-\bm{\alpha}, 0,0}_D)^* \right) [\bm{\phi}](\bm{x}) \mathrm{d}\sigma(\bm{x})
= \int_{\partial D} \bm{\phi}(\bm{x}) \mathrm{d}\sigma(\bm{x}).
\end{align}
This identity plays a crucial role in the proof of Lemma \ref{lem7}.

\begin{proof}[Proof of Lemma \ref{lem7}]
Define $\bm{\varphi}_{js} = (\mathcal{S}_D^{\bm{\alpha},0,0})^{-1}[\bm{e}_{s} \chi_{\partial D_{j}}]$. Using \eqref{f22} and \eqref{f11}, we obtain
\begin{align*}
\bm{C}^{\bm{\alpha},s}_{ij}
&= \int_{\partial D_i} \bm{\varphi}_{js}(\bm{x})  \mathrm{d}\sigma(\bm{x}) \\
&= \int_{\partial D_i} \left( \frac{1}{2} \mathcal{I} + (\mathcal{K}^{-\bm{\alpha},0,0}_D)^* \right) [\bm{\varphi}_{js}](\bm{x})  \mathrm{d}\sigma(\bm{x}) \\
&= \int_{\partial D_i} \frac{\partial}{\partial \bm{\nu}} \mathcal{S}^{\bm{\alpha},0,0}_D [\bm{\varphi}_{js}](\bm{x})  \mathrm{d}\sigma(\bm{x}).
\end{align*}

Set $\bm{u}^{\bm{\alpha}}_{js}(\bm{x}) = \mathcal{S}^{\bm{\alpha},0,0}_D [\bm{\varphi}_{js}](\bm{x})$. Then
\begin{align*}
\bm{C}^{\bm{\alpha},s}_{ij}
&= \int_{\partial D_i} \frac{\partial \bm{u}^{\bm{\alpha}}_{js}(\bm{x})}{\partial \bm{\nu}}  \mathrm{d}\sigma(\bm{x}) \\
&= \sum_{s'=1}^{d} \left( \int_{\partial D_i} \frac{\partial \bm{u}^{\bm{\alpha}}_{js}(\bm{x})}{\partial \bm{\nu}} \cdot \bm{e}_{s'}  \mathrm{d}\sigma(\bm{x}) \right) \bm{e}_{s'},
\end{align*}
where $\bm{u}^{\bm{\alpha}}_{js}$ satisfies the boundary value problem:
\begin{align*}
\begin{cases}
\mathcal{L}^{\lambda,\mu} \bm{u}^{\bm{\alpha}}_{js}(\bm{x}) = \bm{0}, & \bm{x} \in Y \setminus \overline{D}, \\
\bm{u}^{\bm{\alpha}}_{js}(\bm{x}) = \bm{e}_{s} \chi_{\partial D_{j}}(\bm{x}), & \bm{x} \in \partial D, \\
\bm{u}^{\bm{\alpha}}_{js}(\bm{x}) e^{-\mathrm{i}\bm{\alpha}\cdot \bm{x}} \text{ is periodic}.
\end{cases}
\end{align*}
For $\bm{\alpha}$-quasiperiodic functions $\bm{u}, \bm{v} \in H^1(Y \setminus \overline{D})^d$, the following property  holds \cite[p.~160]{2009Layer}:
\begin{align*}
\int_{\partial Y} \frac{\partial \bm{u}}{\partial \bm{\nu}} \cdot \overline{\bm{v}}  \mathrm{d}\sigma = 0.
\end{align*}
Applying integration by parts, for $\bm{u}^{\bm{\alpha}}_{js}(\bm{x}),  \bm{u}^{\bm{\alpha}} _{is'}(\bm{x}) \in H^1(Y\setminus \overline{D})^d$,   we obtain
\begin{align*}
&\int_{\partial D_i}\frac{\partial\bm{u}^{\bm{\alpha}}_{js}(\bm{x})}{\partial\bm{\nu}}\cdot \bm{e}_{s'}\chi_{\partial D_i}(\bm{x}){\rm d}\sigma(\bm{x})\\
=&\int_{\partial D_i}\frac{\partial\bm{u}^{\bm{\alpha}}_{js}(\bm{x})}{\partial\bm{\nu}}\cdot \bm{e}_{s'}\chi_{\partial D_i}(\bm{x}){\rm d}\sigma(\bm{x})
+\sum_{k=1,k\neq i}^N\int_{\partial D_k}\frac{\partial\bm{u}^{\bm{\alpha}}_{js}(\bm{x})}{\partial\bm{\nu}}\cdot \bm{e}_{s'}\chi_{\partial D_i}(\bm{x}){\rm d}\sigma(\bm{x})\\
&\qquad-\int_{\partial Y}\frac{\partial\bm{u}^{\bm{\alpha}}_{js}(\bm{x})}{\partial\bm{\nu}}\cdot \bm{e}_{s'}\chi_{\partial D_i}(\bm{x}){\rm d}\sigma(\bm{x})\\
=&\sum_{k=1,}^N\int_{\partial D_k}\frac{\partial\bm{u}^{\bm{\alpha}}_{js}(\bm{x})}{\partial\bm{\nu}}\cdot\overline{ \bm{u}^{\bm{\alpha}} _{is'}(\bm{x})}{\rm d}\sigma(\bm{x})
-\int_{\partial Y}\frac{\partial\bm{u}^{\bm{\alpha}}_{js}(\bm{x})}{\partial\bm{\nu}}\cdot \overline{ \bm{u}^{\bm{\alpha}}_{is'}(\bm{x})}{\rm d}\sigma(\bm{x})\\
=&-\int_{Y\backslash D}\mathcal{L}^{\lambda,\mu}\bm{u}^{\bm{\alpha}} _{js}(\bm{x})\cdot\overline{ \bm{u}^{\bm{\alpha}} _{is'}(\bm{x})}{\rm d}\bm{x}
-E_{Y\setminus \overline{D}}(\bm{u}^{\bm{\alpha}}_{js},\overline{ \bm{u}^{\bm{\alpha}}_{is'}})\\
=&-\overline{E_{Y\setminus \overline{D}}(\bm{u}^{\bm{\alpha}} _{is'},\overline{ \bm{u}^{\bm{\alpha}} _{js}})}\\
=&\int_{\partial D_j}\overline{\frac{\partial\bm{u}^{\bm{\alpha}}_{is'}(\bm{x})}{\partial\bm{\nu}}\cdot \bm{e}_{s}\chi_{\partial D_j}(\bm{x})}{\rm d}\sigma(\bm{x}),
\end{align*}
where the sesquilinear form $E_{Y \setminus \overline{D}}$ is defined as
\begin{align*}
E_{Y\setminus \overline{D}}(\bm{u}^{\bm{\alpha}}_{js},\overline{ \bm{u}^{\bm{\alpha}}_{is'}})=\int_{Y\setminus \overline{D}}\left[\lambda(\nabla\cdot \bm{u}^{\bm{\alpha}}_{js})(\nabla\cdot \overline{ \bm{u}^{\bm{\alpha}}_{is'}})
+\frac{\mu}{2}(\nabla \bm{u}^{\bm{\alpha}}_{js}+\nabla (\bm{u}^{\bm{\alpha}}_{js})^\top):(\nabla \overline{ \bm{u}^{\bm{\alpha}}_{is'}}+\nabla (\overline{ \bm{u}^{\bm{\alpha}}_{is'}})^\top)\right]{\rm d}\bm{x},
\end{align*}
and $\bm{A}:\bm{B}=\displaystyle\sum_{i,j=1}^{d}a_{ij}b_{ij}$ signifies the Frobenius product between square matrices $\bm{A}=(a_{ij})_{i,j=1}^d$ and $\bm{B}=(b_{ij})_{i,j=1}^d$, as defined in \cite[p.~15]{Ammari2015}.
Consequently, $C_{ij}^{\bm{\alpha},s,s'} = \overline{C_{ji}^{\bm{\alpha},s',s}}$.
\end{proof}

Consider the Neumann boundary value problem:
\begin{align}\label{f13}
\begin{cases}
\mathcal{L}^{\lambda,\mu} \bm{u}(\bm{x}) = \bm{0}, & \bm{x} \in D, \\[2pt]
\displaystyle \frac{\partial \bm{u}(\bm{x})}{\partial \bm{\nu}} = \bm{0}, & \bm{x} \in \partial D, \\[2pt]
\bm{u}(\bm{x}) e^{-\mathrm{i} \bm{\alpha} \cdot \bm{x}} \text{ is periodic}.
\end{cases}
\end{align}
For any solution $\bm{u} \in H^1(D)^d$, the following  equality  holds:
\begin{align*}
\int_{D} \left[ \lambda |\nabla \cdot \bm{u}|^2 + \frac{\mu}{2} \big( \nabla \overline{\bm{u}} + (\nabla \overline{\bm{u}})^\top \big) : \big( \nabla \bm{u} + (\nabla \bm{u})^\top \big) \right] \mathrm{d}\bm{x}
&= \int_{\partial D} \overline{\bm{u}} \cdot \frac{\partial \bm{u}}{\partial \bm{\nu}}  \mathrm{d}\sigma(\bm{x}) 
- \int_{D} \overline{\bm{u}} \cdot \mathcal{L}^{\lambda,\mu} \bm{u}  \mathrm{d}\bm{x} \\
&= 0.
\end{align*}
By the coercivity of the elastic bilinear form, $\bm{u}$ must be constant in $D$; see \cite[p.~160]{2009Layer}. 

This result leads to the following characterization of the kernel:
\begin{lemm}\label{lem3}
Assume that $\bm{\alpha } \neq 0$, and let $\chi_{\partial D_{i}}$ denote the characteristic function of $\partial D_{i}$ for $i = 1, \dots, N$. Then,
\begin{align*}
\ker \left( -\frac{1}{2} \mathcal{I} + (\mathcal{K}^{-\bm{\alpha}, 0,0}_D)^* \right) = \operatorname{span} \Big\{ (\mathcal{S}^{\bm{\alpha},0,0}_D)^{-1} [\bm{e}_{s} \chi_{\partial D_{i}}] \Big\}_{i=1,s=1}^{N, d}.
\end{align*}
\end{lemm}

\begin{proof}
For any $\bm{\phi} \in \ker \left( -\displaystyle\frac{1}{2} \mathcal{I} + (\mathcal{K}^{-\bm{\alpha},0,0}_D)^* \right)$, we have
\begin{align*}
\left( -\frac{1}{2} \mathcal{I} + (\mathcal{K}^{-\bm{\alpha},0,0}_D)^* \right) [\bm{\phi}](\bm{x}) 
= \left. \frac{\partial}{\partial \bm{\nu}} \mathcal{S}^{\bm{\alpha},0,0}_D [\bm{\phi}](\bm{x}) \right|_{-} = 0.
\end{align*}
Define $\bm{u}(\bm{x}) = \mathcal{S}^{\bm{\alpha},0,0}_D [\bm{\phi}](\bm{x})$, which satisfies:
\begin{align*}
\begin{cases}
\mathcal{L}^{\lambda,\mu} \bm{u}(\bm{x}) = \bm{0}, & \bm{x} \in D, \\[2pt]
\displaystyle \frac{\partial \bm{u}(\bm{x})}{\partial \bm{\nu}} = \bm{0}, & \bm{x} \in \partial D, \\[2pt]
\bm{u}(\bm{x}) e^{-\mathrm{i} \bm{\alpha} \cdot \bm{x}} \text{ is periodic}.
\end{cases}
\end{align*}
From the energy argument in \eqref{f13}, $\bm{u}|_D$ is constant on each component $D_i$. Thus,
\begin{align*}
\bm{u}|_{\partial D} = \sum_{i=1}^N \bm{u}|_{\partial D_i} = \sum_{i=1}^N \sum_{s=1}^d k_{is} \bm{e}_s \chi_{\partial D_i}.
\end{align*}
Equivalently, the density function satisfies
\begin{align*}
\bm{\phi}(\bm{x}) = \sum_{i=1}^N \sum_{s=1}^d k_{is} (\mathcal{S}^{\bm{\alpha},0,0}_D)^{-1} [\bm{e}_s \chi_{\partial D_i}](\bm{x}),
\end{align*}
proving 
\[\ker \left( -\frac{1}{2} \mathcal{I} + (\mathcal{K}^{-\bm{\alpha},0,0}_D)^* \right) \subset \operatorname{span} \Big\{ (\mathcal{S}^{\bm{\alpha},0,0}_D)^{-1} [\bm{e}_s \chi_{\partial D_i}] \Big\}_{i=1,s=1}^{N,d}.\]

Conversely, since
\begin{align*}
\frac{\partial}{\partial \bm{\nu}} \mathcal{S}^{\bm{\alpha},0,0}_D \Big[ (\mathcal{S}^{\bm{\alpha},0,0}_D)^{-1} [\bm{e}_s \chi_{\partial D_i}] \Big] = 0,
\end{align*}
all $\bm{\phi} \in \operatorname{span} \left\{ (\mathcal{S}^{\bm{\alpha},0,0}_D)^{-1} [\bm{e}_s \chi_{\partial D_i}] \right\}_{i=1,s=1}^{N,d}$ satisfy
\begin{align*}
\left( -\frac{1}{2} \mathcal{I} + (\mathcal{K}^{-\bm{\alpha},0,0}_D)^* \right) [\bm{\phi}] = 0.
\end{align*}
Thus the reverse inclusion holds, completing the proof.
\end{proof}

\begin{proof}[Proof of Theorem \ref{th4}]
Let $\bm{\phi}, \bm{\psi} \in H^{-1/2}(\partial D)^d$. The solution to \eqref{f8} can be expressed as
\begin{align*}
\bm{v}(\bm{x}) = 
\begin{cases}
\mathcal{S}^{\bm{\alpha},\bm{\rho},\omega}_D[\bm{\psi}](\bm{x}), & \bm{x} \in Y \setminus \overline{D}, \\
\mathcal{S}^{\bm{\alpha},\bm{\rho},\frac{\delta}{\varepsilon}\omega}_D[\bm{\phi}](\bm{x}), & \bm{x} \in D.
\end{cases}
\end{align*}
Continuity across $\partial D$ implies:
\begin{align*}
\bm{v}(\bm{x})|_{+} = \bm{v}(\bm{x})|_{-} \quad \Leftrightarrow \quad \mathcal{S}^{\bm{\alpha},\bm{\rho},\omega}_D[\bm{\psi}](\bm{x}) = \mathcal{S}^{\bm{\alpha},\bm{\rho},\frac{\delta}{\varepsilon}\omega}_D[\bm{\phi}](\bm{x}), \quad \bm{x} \in \partial D.
\end{align*}
Using the asymptotic expansions in \eqref{f9-} and Lemma \ref{lem2} as $\omega \to 0$, we obtain:
\begin{align}\label{f15}
\left( \mathcal{S}^{\bm{\alpha},0,0}_D + \sum_{k=1}^{\infty} \omega^{2k} \mathcal{S}^{\bm{\alpha},\bm{\rho},0}_{D,k} \right)[\bm{\psi}](\bm{x})
&= \left( \mathcal{S}^{\bm{\alpha},\bm{\rho},0}_D + \sum_{k=1}^{\infty} \left( \frac{\delta}{\varepsilon} \right)^k \omega^{2k} \mathcal{S}^{\bm{\alpha},\bm{\rho},0}_{D,k} \right)[\bm{\phi}](\bm{x}) \nonumber\\
\mathcal{S}^{\bm{\alpha},0,0}_D[\bm{\psi}](\bm{x}) &= \mathcal{S}^{\bm{\alpha},0,0}_D[\bm{\phi}](\bm{x}) + \mathcal{O}\left( \omega^2 + \frac{\delta}{\varepsilon} \omega^2 \right) \nonumber \\
\bm{\psi}(\bm{x}) &= \bm{\phi}(\bm{x}) + \mathcal{O}(\omega^2).
\end{align}

The flux continuity condition $\displaystyle\delta \frac{\partial \bm{v}}{\partial \bm{\nu}}\Big|_{+} = \frac{\partial \bm{v}}{\partial \bm{\nu}}\Big|_{-}$ combined with the jump relation \eqref{f11} yields:
\begin{align*}
\delta \left( \frac{1}{2} \mathcal{I} + (\mathcal{K}^{-\bm{\alpha},\bm{\rho},\omega}_D)^* \right)[\bm{\psi}](\bm{x})
= \left( -\frac{1}{2} \mathcal{I} + (\mathcal{K}^{-\bm{\alpha},\bm{\rho},\frac{\delta}{\varepsilon}\omega}_D)^* \right)[\bm{\phi}](\bm{x}), \quad \bm{x} \in \partial D.
\end{align*}

Applying  the asymptotic expansions in \eqref{f10-},  we derive
\begin{align}\label{f16}
\left( -\frac{1}{2} \mathcal{I} + (\mathcal{K}^{-\bm{\alpha}, 0,0}_D)^* \right)[\bm{\phi}](\bm{x}) 
&+ \frac{\delta}{\varepsilon} \omega^{2} \mathcal{K}^{\bm{\alpha},\bm{\rho},0}_{D,1} [\bm{\phi}](\bm{x}) \nonumber\\
&- \delta \left( \frac{1}{2} \mathcal{I} + (\mathcal{K}^{-\bm{\alpha}, 0,0}_D)^* \right)[\bm{\psi}](\bm{x}) 
= \mathcal{O}\left( \delta \omega^2 + \frac{\delta^2}{\varepsilon^2} \omega^{4} \right).
\end{align}
Using the last equality in \eqref{f15}, \eqref{f16} can be rewritten as 
\begin{align}\label{f16-1}
\left( -\frac{1}{2} \mathcal{I} + (\mathcal{K}^{-\bm{\alpha}, 0,0}_D)^* \right)
&[\bm{\phi}](\bm{x}) 
+ \frac{\delta}{\varepsilon} \omega^{2} \mathcal{K}^{\bm{\alpha},\bm{\rho},0}_{D,1} [\bm{\phi}](\bm{x}) \nonumber\\
&- \delta \left( \frac{1}{2} \mathcal{I} + (\mathcal{K}^{-\bm{\alpha}, 0,0}_D)^* \right)[\bm{\phi}](\bm{x}) - \mathcal{O}\left( \delta \omega^2 \right)
= \mathcal{O}\left( \delta \omega^2 + \frac{\delta^2}{\varepsilon^2} \omega^{4} \right).
\end{align}

Following Lemma \ref{lem3}, we can decompose $\bm{\phi}$ as:
\begin{align*}
\bm{\phi}(\bm{x}) = \sum_{i=1}^N \sum_{s=1}^d k_{is} (\mathcal{S}^{\bm{\alpha},0,0}_D)^{-1} [\bm{e}_{s} \chi_{\partial D_i}] + \widetilde{\bm{\phi}}(\bm{x}),
\end{align*}
where $ \widetilde{\bm{\phi}}(\bm{x}) $  belongs to  the orthogonal complement of  $\operatorname{span} \Big\{ (\mathcal{S}^{\bm{\alpha},0,0}_D)^{-1} [\bm{e}_{s} \chi_{\partial D_{i}}] \Big\}_{i=1,s=1}^{N, d}$.
Substituting into \eqref{f16-1} gives:
\begin{align*}
\left( -\frac{1}{2} \mathcal{I} + (\mathcal{K}^{-\bm{\alpha}, 0,0}_D)^* \right)[\widetilde{\bm{\phi}}](\bm{x}) &= \mathcal{O}\left( \delta + \frac{\delta}{\varepsilon} \omega^2 \right), \\
\widetilde{\bm{\phi}}(\bm{x}) &= \mathcal{O}\left( \delta + \frac{\delta}{\varepsilon} \omega^2 \right).
\end{align*}
Thus 
\[\bm{\phi}(\bm{x}) = \sum_{i=1}^N \sum_{s=1}^d k_{is} (\mathcal{S}^{\bm{\alpha},0,0}_D)^{-1} [\bm{e}_{s} \chi_{\partial D_i}] + \mathcal{O}\left( \delta + \frac{\delta}{\varepsilon} \omega^2 \right).\]
Now substituting above  into \eqref{f16-1}  and integrating  over $\partial D_j$,  using \eqref{f19}, \eqref{f22}, and Lemma \ref{lem4}, we obatin
\begin{align*}
&-\frac{\delta}{\varepsilon} \omega^{2} \bm{\rho} \sum_{i=1}^N \sum_{s=1}^d k_{is} \int_{\partial D_j} \bm{e}_{s} \chi_{\partial D_i}  \mathrm{d}\sigma(\bm{x}) 
- \delta \sum_{i=1}^N \sum_{s=1}^d k_{is} \int_{\partial D_j} (\mathcal{S}^{\bm{\alpha},0,0}_D)^{-1} [\bm{e}_{s} \chi_{\partial D_i}]  \mathrm{d}\sigma(\bm{x}) \\
&= \mathcal{O}\left( \delta^2 + \delta \omega^2 + \frac{\delta^2}{\varepsilon} \omega^2 + \frac{\delta^2}{\varepsilon^2} \omega^{4} \right).
\end{align*}
Applying \eqref{f18} yields:
\begin{align*}
-\frac{1}{\varepsilon} |D_j| \omega^{2} \bm{\rho} \sum_{s=1}^d k_{js} \bm{e}_s = \sum_{i=1}^N \sum_{s=1}^d k_{is} \bm{C}^{\bm{\alpha},s}_{ji} + \mathcal{O}\left( \delta + \omega^2 + \frac{\delta}{\varepsilon} \omega^2 + \frac{\delta}{\varepsilon^2} \omega^{4} \right).
\end{align*}

This system is compactly expressed as:
\begin{align*}
\omega^2 \bm{\vartheta} = -\varepsilon \bm{H} \bm{\vartheta} + \mathcal{O}\left( \delta \varepsilon + \varepsilon \omega^2 + \delta \omega^2 + \frac{\delta}{\varepsilon} \omega^{4} \right),
\end{align*}
where $\bm{\vartheta} = (k_{11}, \dots, k_{1d}, \dots, k_{N1}, \dots, k_{Nd})^\top \neq \bm{0}$, and the $dN \times dN$ block matrix $\bm{H}$  and its entries are  defined by 
\begin{align*}
\bm{H} =
\begin{pmatrix}
\bm{H}_{11} & \cdots & \bm{H}_{1N} \\
\vdots & \ddots & \vdots \\
\bm{H}_{N1} & \cdots & \bm{H}_{NN}
\end{pmatrix}, 
\quad 
\bm{H}_{ij} = \frac{1}{|D_i|} \bm{\rho}^{-1} \begin{pmatrix} \bm{C}^{\alpha,1}_{ij} & \bm{C}^{\alpha,2}_{ij} & \cdots & \bm{C}^{\alpha,d}_{ij} \end{pmatrix}.
\end{align*}

The eigenvalues $\xi_i$ of $\bm{H}$ then satisfy:
\begin{align*}
\omega_i = \sqrt{ -\varepsilon \xi_i } + \mathcal{O}(\varepsilon), \quad i = 1, \cdots, dN.
\end{align*}
\end{proof}

\section{Resonator-modulated system}\label{sec4}

In this section, we analyze the $d$-dimensional elastic wave equation. 
Building on the scaling $\omega = \mathcal O(\varepsilon^{1/2})$ established for non-modulated systems, we investigate a configuration where time modulation is confined to the resonators, while the surrounding medium remains time-independent. 
Drawing on \eqref{f46} and \eqref{f7-}, we specifically examine density modulation with Lamé parameters and mass densities satisfying
\begin{align*}
\bm{\rho}(\bm{x},t) =
\begin{cases}
\bm{\rho}, & \bm{x} \in \mathbb{R}^{d} \setminus \overline{\mathcal{C}}, \\
\widetilde{\bm{\rho}} \widetilde{\bm{\rho}}_{i}(t), & \bm{x} \in \mathcal{C}_{i},
\end{cases} \quad 
(\widehat{\lambda}(\bm{x},t), \widehat{\mu}(\bm{x},t)) =
\begin{cases}
(\lambda, \mu), & \bm{x} \in \mathbb{R}^{d} \setminus \overline{\mathcal{C}}, \\
(\widetilde{\lambda}, \widetilde{\mu}), & \bm{x} \in \mathcal{C},
\end{cases}
\end{align*}
where $\widetilde{\bm{\rho}}_{i}(t) = \operatorname{diag}\big( \widetilde{\rho}_{i1}(t), \cdots, \widetilde{\rho}_{id}(t) \big)$ is $T$-periodic.

Given that $\bm{U}(\bm{x},t)e^{-\mathrm{i}\omega t}$ is $T$-periodic, we employ the Fourier series expansions
\begin{align*}
\bm{U}(\bm{x},t) &= e^{\mathrm{i}\omega t} \sum_{n=-\infty}^{\infty} \bm{v}_{n}(\bm{x}) e^{\mathrm{i} n \Omega t}, \\
\widetilde{\bm{\rho}}_{i}(t) &= \operatorname{diag}\left( \sum_{n=-\infty}^{\infty} a^1_{in} e^{\mathrm{i} n \Omega t}, \dots, \sum_{n=-\infty}^{\infty} a^d_{in} e^{\mathrm{i} n \Omega t} \right),
\end{align*}
where $\bm{v}_{n}(\bm{x}) = \big( v^1_{n}(\bm{x}), \cdots, v^d_{n}(\bm{x}) \big)^\top$.
Equation \eqref{f5} is then equivalent to the following system for all $n \in \mathbb{Z}$:
\begin{align}\label{f20}
\begin{cases}
\mathcal{L}^{\lambda, \mu} \bm{v}_{n}(\bm{x}) + \bm{\rho} (\omega + n \Omega)^2 \bm{v}_{n}(\bm{x}) = \bm{0}, & \bm{x} \in Y \setminus \overline{D}, \\[2pt]
\mathcal{L}^{\lambda, \mu} \bm{v}_{n}(\bm{x}) + \dfrac{\delta}{\varepsilon} \bm{\rho} \bm{v}_{in}(\bm{x}) = \bm{0}, & \bm{x} \in D_i, \\[2pt]
\bm{v}_{n}\big|_{+} = \bm{v}_{n}\big|_{-}, & \bm{x} \in \partial D, \\[2pt]
\delta \dfrac{\partial \bm{v}_{n}}{\partial \bm{\nu}}\bigg|_{+} = \dfrac{\partial \bm{v}_{n}}{\partial \bm{\nu}}\bigg|_{-}, & \bm{x} \in \partial D, \\[2pt]
\bm{v}_{n}(\bm{x}) e^{-\mathrm{i}\bm{\alpha} \cdot \bm{x}} \text{ is periodic},
\end{cases}
\end{align}
with the coupling term 
\begin{align*} 
\bm{v}_{in}(\bm{x}) = \left( \sum_{m=-\infty}^{\infty} a^1_{im} (\omega + (n-m)\Omega)^2 v^1_{n-m}(\bm{x}), \cdots, \sum_{m=-\infty}^{\infty} a^d_{im} (\omega + (n-m)\Omega)^2 v^d_{n-m}(\bm{x}) \right)^\top.
\end{align*}

Assuming the solution is normalized such that $\|\bm{v}_0\|_{H^1(Y)^d} = 1$, and leveraging the continuous differentiability of $\bm{U}(\bm{x},t)$ in $t$ with
\begin{align*}
\partial_t \bm{U}(\bm{x},t) = \sum_{n=-\infty}^{\infty} \mathrm{i}(\omega + n\Omega)\bm{v}_n(\bm{x}) e^{\mathrm{i}(\omega + n\Omega)t},
\end{align*}
we obtain the asymptotic decay
\begin{align}\label{f37}
\|\bm{v}_n\|_{H^1(Y)^d} = o\left(\frac{1}{|n|}\right) \quad \text{as} \quad |n| \to \infty.
\end{align}

We now consider modulated densities with finite Fourier representations:
\begin{align*}
\widetilde{\bm{\rho}}_{i}(t) = \operatorname{diag}\left( \sum_{n=-M}^{M} a^1_{in} e^{\mathrm{i}n\Omega t},\ \dots,\ \sum_{n=-M}^{M} a^d_{in} e^{\mathrm{i}n\Omega t} \right),
\end{align*}
where $M = \mathcal{O}(\varepsilon^{-\gamma/2})$ for $0 < \gamma < 1$. Our objective is to determine subwavelength quasi-frequencies $\omega$ satisfying \eqref{f20}.
Following Theorem \ref{th4} and Definition \ref{defi1}, we impose the asymptotic orders:
\begin{align*}
\omega = \mathcal{O}(\varepsilon^{1/2}), \quad \Omega = \mathcal{O}(\varepsilon^{1/2}).
\end{align*}
This yields the following  result.
\begin{lemm}\label{lem5}
As $\delta, \varepsilon \to 0$, the functions $\bm{v}_{n}(\bm{x})$ for $|n| \leq M$ remain approximately stable, satisfying 
\begin{align*}
\bm{v}_{n}(\bm{x}) = \bm{c}_{in} + \mathcal{O}(\varepsilon^{(1-\gamma)/2}), \quad \bm{x} \in D_i, \ i = 1, \dots, N,
\end{align*}
where $\bm{c}_{in}$ are constant vectors in $\mathbb{R}^d$.
\end{lemm}

\begin{proof}
For $\bm{v}_{n} \in H^1(D_i)^d$, the  identity gives:
\begin{align*}
\int_{D_i} \Bigl[\lambda |\nabla \cdot \bm{v}_n|^2 &+ \frac{\mu}{2} \big( \nabla \overline{\bm{v}_n} + (\nabla \overline{\bm{v}_n})^\top \big) : \big( \nabla \bm{v}_n + (\nabla \bm{v}_n)^\top \big) \Bigl]  \mathrm{d}\bm{x} \\
&= -\int_{D_i} \overline{\bm{v}_n} \cdot \mathcal{L}^{\lambda,\mu} \bm{v}_n  \mathrm{d}\bm{x} 
+ \int_{\partial D_i} \overline{\bm{v}_n} \cdot \left. \frac{\partial \bm{v}_n}{\partial \bm{\nu}} \right|_{-} \mathrm{d}\sigma(\bm{x}).
\end{align*}
Using the governing equation from \eqref{f20} and the decay estimate \eqref{f37}, we bound each term
\begin{align*}
-\int_{D_i} \overline{\bm{v}_n} \cdot \mathcal{L}^{\lambda,\mu} \bm{v}_n  \mathrm{d}\bm{x} 
&= \frac{\delta}{\varepsilon} \int_{D_i} \overline{\bm{v}_n} \cdot \bm{\rho} \bm{v}_{in}(\bm{x})  \mathrm{d}\bm{x} 
= \mathcal{O}(\varepsilon^{1-\gamma}), \\
\int_{\partial D_i} \overline{\bm{v}_n} \cdot \left. \frac{\partial \bm{v}_n}{\partial \bm{\nu}} \right|_{-} \mathrm{d}\sigma(\bm{x})
&= \delta \int_{\partial D_i} \overline{\bm{v}_n} \cdot \left. \frac{\partial \bm{v}_n}{\partial \bm{\nu}} \right|_{+} \mathrm{d}\sigma(\bm{x})
= \mathcal{O}(\delta),
\end{align*}
where the boundary term estimate follows from \cite{Zhang2025}. Combining these yields
\begin{align*}
\int_{D_i} \left[ \lambda |\nabla \cdot \bm{v}_n|^2 + \frac{\mu}{2} \big( \nabla \overline{\bm{v}_n} + (\nabla \overline{\bm{v}_n})^\top \big) : \big( \nabla \bm{v}_n + (\nabla \bm{v}_n)^\top \big) \right] \mathrm{d}\bm{x} = \mathcal{O}(\varepsilon^{1-\gamma} + \delta).
\end{align*}
By Korn's inequality, there exists $C_K > 0$ such that
\begin{align*}
\|\nabla \bm{v}_n\|_{L^2(D_i)}^2 \leq C_K \int_{D_i} \big| \nabla \bm{v}_n + (\nabla \bm{v}_n)^\top \big|^2 \mathrm{d}\bm{x} = \mathcal{O}(\varepsilon^{1-\gamma} + \delta).
\end{align*}
Applying the Poincar\'{e}-Wirtinger inequality to $\bm{v}_n - \bm{c}_{in}$ where $\bm{c}_{in} =\displaystyle \frac{1}{|D_i|} \int_{D_i} \bm{v}_n \mathrm{d}\bm{x}$, we obtain
\begin{align*}
\|\bm{v}_n - \bm{c}_{in}\|_{L^2(D_i)} \leq C_P \|\nabla \bm{v}_n\|_{L^2(D_i)} = \mathcal{O}(\varepsilon^{(1-\gamma)/2} + \delta^{1/2}).
\end{align*}
Finally, under the asymptotic condition  $\delta \leq  \mathcal{O}(\varepsilon)$ ( see \eqref{relation}), we conclude for $\bm{x} \in D_i$ and $|n| \leq M$:
\begin{align*}
\bm{v}_n(\bm{x}) = \bm{c}_{in} + \mathcal{O}(\varepsilon^{(1-\gamma)/2}),
\end{align*}
where $\bm{c}_{in} \in \mathbb{R}^d$ is constant on each resonator $D_i$.

\end{proof}

Building on the preceding lemmas and layer potential analysis, we derive a system of ODEs that asymptotically characterizes the subwavelength quasi-frequencies of \eqref{f2}. This result provides the foundation for constructing EPs via Floquet theory in the following section.

\begin{theo}\label{th5}
Assume $\bm{\alpha} \neq \bm{0}$ and consider the joint asymptotic limit $\delta, \varepsilon \to 0$. In the subwavelength regime, the quasi-frequencies $\omega\in Y^*_{t}$ of the equation \eqref{f20} are predominantly determined by the quasi-frequencies of the system of ODEs
\begin{align}\label{f26}
\sum_{j=1}^N\sum_{s=1}^d y^{\omega}_{js}(t)\bm{C}^{\bm{\alpha},s}_{ij}
=\frac{|D_i|}{\varepsilon}\bm{\rho}\widetilde{\bm{\rho}}_i(t)\left(\sum_{s=1}^d\frac{\mathrm{d}^2 y^{\omega}_{is}(t)}{\mathrm{d}t^2}\bm{e}_s\right), \quad i = 1, \dots, N,
\end{align}
where $\widetilde{\bm{\rho}}_i(t)$ denotes the modulated density in resonator $\mathcal C_i$.

\end{theo}

\begin{proof}
For each $n \in \{-M, \dots, M\}$ and $\bm{x} \in Y \setminus \overline{D}$, the solution $\bm{v}_{n}(\bm{x})$ admits the single-layer potential representation
\begin{align}\label{f23}
\bm{v}_{n}(\bm{x}) = \mathcal{S}^{\bm{\alpha},\bm{\rho},(\omega+n\Omega)}_{D}[\bm{\phi}_{n}](\bm{x}), \quad \bm{\phi}_{n} \in H^{-1/2}(\partial D)^d.
\end{align}
Integrating the interior equation over $D_i$ yields
\begin{align}\label{f25}
\int_{D_i} \mathcal{L}^{\lambda, \mu} \bm{v}_{n}(\bm{x})  \mathrm{d}\bm{x} = -\frac{\delta}{\varepsilon}\bm{\rho} \int_{D_i} \bm{v}_{in}(\bm{x})  \mathrm{d}\bm{x}.
\end{align}
First consider the left-hand side of \eqref{f25}. Using \eqref{f23}, \eqref{f11}, and \eqref{f22}, we obtain
\begin{equation}
\begin{aligned}\label{f24}
\int_{D_i} \mathcal{L}^{\lambda, \mu} \bm{v}_{n}(\bm{x})  \mathrm{d}\bm{x}
&= \int_{\partial D_i} \frac{\partial \bm{v}_{n}}{\partial \bm{\nu}}(\bm{x})  \mathrm{d}\sigma(\bm{x}) \\
&= \delta \int_{\partial D_i} \frac{\partial \bm{v}_{n}}{\partial \bm{\nu}}\Big|_+(\bm{x})  \mathrm{d}\sigma(\bm{x}) \\
&= \delta \int_{\partial D_i} \frac{\partial}{\partial \bm{\nu}} \mathcal{S}^{\bm{\alpha},\bm{\rho},(\omega+n\Omega)}_{D}[\bm{\phi}_{n}](\bm{x})\Big|_+  \mathrm{d}\sigma(\bm{x}) \\
&= \delta \int_{\partial D_i} \left( \frac{1}{2}\mathcal{I} + \left(\mathcal{K}^{-\bm{\alpha},0,0}\right)^* \right)[\bm{\phi}_{n}](\bm{x})  \mathrm{d}\sigma(\bm{x}) + \mathcal{O}(\delta \varepsilon^{1-\gamma}) \\
&= \delta \int_{\partial D_i} \bm{\phi}_{n}(\bm{x})  \mathrm{d}\sigma(\bm{x}) + \mathcal{O}(\delta \varepsilon^{1-\gamma}).
\end{aligned}
\end{equation}
By Lemma \ref{lem5}, the interior trace satisfies
\begin{align*}
\bm{v}_{n}\big|_-(\bm{x}) 
= \sum_{j=1}^{N} \bm{v}_{n}\big|_{-} \chi_{\partial D_j} = \sum_{j=1}^{N} \left( \sum_{s=1}^d k_{jsn} \bm{e}_s \right) \chi_{\partial D_j}(\bm{x}),
\end{align*}
where $k_{jsn}$ are constants. The trace continuity condition $\bm{v}_{n}\big|_+ = \bm{v}_{n}\big|_-$ on $\partial D$ implies
\begin{align*}
\mathcal{S}^{\bm{\alpha},0,0}[\bm{\phi}_n](\bm{x}) + \mathcal{O}(\varepsilon^{1-\gamma}) 
= \sum_{j=1}^N \sum_{s=1}^d k_{jsn} \bm{e}_s \chi_{\partial D_j}(\bm{x}), \quad \bm{x} \in \partial D.
\end{align*}
Thus \eqref{f24} becomes
\begin{align*}
\int_{D_i} \mathcal{L}^{\lambda, \mu} \bm{v}_{n}(\bm{x})  \mathrm{d}\bm{x}
&= \delta \sum_{j=1}^N \sum_{s=1}^d k_{jsn} \int_{\partial D_i} \left( \mathcal{S}^{\bm{\alpha},0,0} \right)^{-1} \left[ \bm{e}_s \chi_{\partial D_j} \right] (\bm{x})\mathrm{d}\sigma(\bm{x}) + \mathcal{O}(\delta \varepsilon^{1-\gamma}) \\
&= \delta \sum_{j=1}^N \sum_{s=1}^d k_{jsn} \bm{C}^{\bm{\alpha},s}_{ij} + \mathcal{O}(\delta \varepsilon^{1-\gamma}).
\end{align*}
Now consider the right-hand side of \eqref{f25}:
\begin{align*}
\frac{\delta}{\varepsilon}\bm{\rho} \int_{D_i} \bm{v}_{in}(\bm{x})  \mathrm{d}\bm{x}
&= \frac{\delta}{\varepsilon}\bm{\rho} \begin{pmatrix}
\displaystyle\sum_{m=-M}^{M} a^1_{im} (\omega + (n - m)\Omega)^2 \int_{D_i} v^1_{n-m}(\bm{x})  \mathrm{d}\bm{x} \\ 
\vdots \\
\displaystyle\sum_{m=-M}^{M} a^d_{im} (\omega + (n - m)\Omega)^2 \int_{D_i} v^d_{n-m}(\bm{x})  \mathrm{d}\bm{x}
\end{pmatrix} \\
&= \frac{\delta}{\varepsilon}\bm{\rho} \begin{pmatrix} 
\displaystyle\sum_{m=-M}^{M} a^1_{im} (\omega + (n - m)\Omega)^2 k_{i1(n-m)} |D_i| \\ 
\vdots \\
\displaystyle\sum_{m=-M}^{M} a^d_{im} (\omega + (n - m)\Omega)^2 k_{id(n-m)} |D_i| 
\end{pmatrix}.
\end{align*}
Equating both sides via \eqref{f25} gives
\begin{align*}
\sum_{j=1}^N \sum_{s=1}^d k_{jsn} \bm{C}^{\bm{\alpha},s}_{ij} = -\frac{|D_i|}{\varepsilon} \bm{\rho} \begin{pmatrix} 
\displaystyle\sum_{m=-M}^{M} a^1_{im} (\omega + (n - m)\Omega)^2 k_{i1(n-m)} \\ 
\vdots \\
\displaystyle\sum_{m=-M}^{M} a^d_{im} (\omega + (n - m)\Omega)^2 k_{id(n-m)} 
\end{pmatrix} + \mathcal{O}(\varepsilon^{1-\gamma}).
\end{align*}
Multiplying by $e^{\mathrm{i}(\omega + n\Omega)t}$ and summing over $n \in \{-M, \dots, M\}$ yields
\begin{align*}
\sum_{j=1}^N \sum_{s=1}^d \sum_{n=-M}^M k_{jsn} e^{\mathrm{i}(\omega + n\Omega)t} \bm{C}^{\bm{\alpha},s}_{ij}
= \frac{|D_i|}{\varepsilon} \bm{\rho} \widetilde{\bm{\rho}}_i(t) \sum_{s=1}^d \frac{\mathrm{d}^2}{\mathrm{d}t^2} \left( \sum_{n=-M}^M k_{isn} e^{\mathrm{i}(\omega + n\Omega)t} \right) \bm{e}_s + \mathcal{O}(\varepsilon^{1-\gamma}).
\end{align*}
Define the quasi-periodic functions
\begin{align}
y^{\omega}_{is}(t) = \sum_{n=-\infty}^{\infty} k_{isn} e^{\mathrm{i}(\omega + n\Omega) t},
\end{align}
which satisfy $y^{\omega}_{is}(t + T) = e^{\mathrm{i} \omega T} y^{\omega}_{is}(t)$. By Definition \ref{defi1}, we have the approximation
\begin{align}
y^{\omega}_{is}(t) = \sum_{n=-M}^M k_{isn} e^{\mathrm{i}(\omega + n\Omega) t} + o(1).
\end{align}
Thus we obtain
\begin{align*}
\sum_{j=1}^N \sum_{s=1}^d y^{\omega}_{js}(t) \bm{C}^{\bm{\alpha},s}_{ij} 
= \frac{|D_i|}{\varepsilon} \bm{\rho} \widetilde{\bm{\rho}}_i(t) \sum_{s=1}^d \frac{\mathrm{d}^2 y^{\omega}_{is}(t)}{\mathrm{d}t^2} \bm{e}_s + o(1).
\end{align*}
\end{proof}

Furthermore, equation \eqref{f26} can be expressed as the second-order matrix differential equation
\begin{align}\label{f29}
\frac{\mathrm{d}^2}{\mathrm{d}t^2} \bm{y}(t) = \bm{B}(t) \bm{y}(t),
\end{align}
where $\bm{B}(t)$ is a $dN \times dN$ block matrix defined by
\begin{align*}
\bm{B}(t) &= 
\begin{pmatrix}
\bm{B}_{11}(t) & \cdots & \bm{B}_{1N}(t) \\
\vdots & \ddots & \vdots \\
\bm{B}_{N1}(t) & \cdots & \bm{B}_{NN}(t)
\end{pmatrix},
\end{align*}
with each $d \times d$ submatrix $\bm{B}_{ij}(t)$ given by
\begin{align*}
\bm{B}_{ij}(t) &= \frac{\varepsilon}{|D_i|} \left( \bm{\rho} \widetilde{\bm{\rho}}_{i}(t) \right)^{-1} 
\begin{pmatrix} 
\bm{C}^{\bm{\alpha},1}_{ij} & \bm{C}^{\bm{\alpha},2}_{ij} & \cdots & \bm{C}^{\bm{\alpha},d}_{ij} 
\end{pmatrix}.
\end{align*}

\section{Asymptotic Exceptional Points}\label{sec5}

Inspired by \cite{MR4389752}, this section employs Floquet's theory to develop a method for constructing first-order asymptotic EPs in elastic systems, demonstrating the approach with several examples.

Throughout this section, for a $T$-periodic function $g(t)$, we denote its $m$-th Fourier coefficient by $g^{(m)}$.

\begin{defi}[Exceptional point \cite{MR4389752}]
Let $\bm{A}(c, \cdot) : \mathbb{R} \to \mathbb{C}^{n \times n}$ be a parametrized $T$-periodic matrix-valued function. An element $c_0$ in the parameter space is called an \emph{exceptional point} of the parametrized differential equation
\begin{align*}
\frac{\mathrm{d}\bm{x}}{\mathrm{d}t} = \bm{A}(c,t) \bm{x}
\end{align*}
if the fundamental solution $\bm{X}(c_0,t)$ of $$\frac{\mathrm{d}\bm{x}}{\mathrm{d}t} = \bm{A}(c_0,t)\bm{x}$$ satisfies that $\bm{X}(c_0,T)$ is not diagonalizable.
\end{defi}

According to the  Floquet theory, when the system parameter is at the EP $c_0$, the fundamental solution admits the Floquet representation
\begin{align*}
\bm{X}(c_0,t) = \bm{P}(c_0,t) \exp\left(\bm{F}(c_0) t\right),
\end{align*}
where $\bm{P}(c_0,t)$ is $T$-periodic with $\bm{P}(c_0,0) = \bm{I}_n$. 
After one period $T$, the monodromy matrix simplifies to
\begin{align*}
\bm{X}(c_0,T) = \exp\left(\bm{F}(c_0) T\right).
\end{align*}
At an EP, $\bm{X}(c_0, T)$ is non-diagonalizable and therefore similar to a Jordan matrix containing at least one Jordan block of size $\geq 2$. 

Consider a $2 \times 2$ Jordan block
\begin{align*}
\bm{J}(\sigma) = 
\begin{pmatrix}
\sigma & 1 \\
0 & \sigma
\end{pmatrix}.
\end{align*}
Its matrix exponential is given by
\begin{align*}
\exp\left(\bm{J}(\sigma) t\right) = e^{\sigma t}
\begin{pmatrix}
1 & t \\
0 & 1
\end{pmatrix}.
\end{align*}
This functional form confirms that $\bm{F}(c_0)$ must be non-diagonalizable at the EP.

In this section, we study the family of parameter-dependent ODEs
\begin{align}\label{f28}
    \begin{cases}
        \dfrac{\mathrm{d}\bm{x}_{\eta}}{\mathrm{d}t} = \bm{A}_{\eta}(c,t) \bm{x}_{\eta}, \\[2ex]
        \bm{x}_{\eta}(0) = \bm{x}_{\eta,0},
    \end{cases}
\end{align}
where $\bm{A}_{\eta}(c,\cdot): \mathbb{R} \to \mathbb{C}^{n \times n}$ is a $T$-periodic matrix-valued function depending on parameter $c$, with $\eta \ll 1$ being a small perturbation parameter. The matrix function admits an asymptotic expansion in $\eta$:
\begin{align}\label{f27-}
    \bm{A}_{\eta}(c,t) = \bm{A}_{0}(c) + \eta \bm{A}_{1}(c,t) + \eta^{2} \bm{A}_{2}(c,t) + \cdots,
\end{align}
converging for $|\eta| < r_0$ with $r_0 > 0$ independent of $t$. Following \cite{MR4389752}, we assume:
\begin{enumerate}[label=(\arabic*)]
  \item $\bm{A}_0(c)$ is a constant diagonal matrix;
  \item $\bm{A}_n(c,t)$ ($n \geq 1$) has a finite Fourier series expansion;
  \item The expansion \eqref{f27-} converges uniformly in $t$ for $|\eta| < r_0$.
\end{enumerate}

The fundamental solution $\bm{X}_\eta(c,t)$ of \eqref{f28} admits a Floquet decomposition 
\begin{align*}
\bm{X}_\eta(c,t) = \bm{P}_\eta(c,t) \exp\left(\bm{F}_\eta(c) t\right),
\end{align*}
with both $\bm{P}_\eta(c,t)$ and $\bm{F}_\eta(c)$ analytic in $\eta$. This implies the asymptotic expansions
\begin{align*}
\bm{F}_\eta(c) &= \bm{F}_0(c) + \eta \bm{F}_1(c) + \eta^2 \bm{F}_2(c) + \cdots, \\
\bm{P}_\eta(c,t) &= \bm{P}_0(c,t) + \eta \bm{P}_1(c,t) + \eta^2 \bm{P}_2(c,t) + \cdots,
\end{align*}
where each coefficient matrix is expressible in terms of $\{\bm{A}_k(c,t)\}_{k=0}^n$. The leading-order terms satisfy:
\begin{align*}
(\bm{A}_0(c))_{ii} &= (\bm{F}_0(c))_{ii} + \mathrm{i} \frac{2\pi}{T} n_i, \quad n_i \in \mathbb{Z}, \\
\Im\left((\bm{F}_0(c))_{ii}\right) &\in \left(-\frac{\pi}{T}, \frac{\pi}{T}\right], \\
(\bm{F}_1(c))_{ij} &= \left(\bm{A}_1^{(n_i - n_j)}(c)\right)_{ij}, \quad \text{when } (\bm{F}_0(c))_{ii} = (\bm{F}_0(c))_{jj},
\end{align*}
where $n_i$ are folding numbers and $\bm{A}_1^{(m)}(c)$ denotes the $m$-th Fourier coefficient.
Therefore, at an EP $c_0$, the Floquet exponent matrix 
$$\bm{F}_\eta(c_0) = \sum_{m=0}^{\infty} \eta^m \bm{F}_m(c_0)$$ is non-diagonalizable. This motivates the definition of asymptotic EPs. Consider the system at $c = c_0$:
\begin{align}\label{f50}
\begin{cases}
\dfrac{\mathrm{d}\bm{x}_{\eta}}{\mathrm{d}t} = \bm{A}_{\eta}(c_0,t) \bm{x}_{\eta},\\[2ex]
\bm{x}_{\eta}(0) = \bm{x}_{\eta,0}.
\end{cases}
\end{align}

\begin{defi}[$m$-th order asymptotic exceptional point \cite{MR4389752}]
Suppose $\bm{A}_\eta(t)$ satisfies the assumptions stated above. Then the system \eqref{f50} is at an $m$-th order asymptotic exceptional point $c_0$ if the truncated Floquet exponent
\begin{align*}
\bm{F}^{(m)}_\eta(c_0) = \bm{F}_0(c_0) + \eta \bm{F}_1(c_0) + \cdots + \eta^m \bm{F}_m(c_0)
\end{align*}
is non-diagonalizable for all sufficiently small $\eta > 0$.	
\end{defi}

The literature \cite{doi:10.1137/21M1449427} presents an effective method for computing the eigenvalues of the Floquet exponent matrix $\bm{F}_\eta = \bm{F}_0 + \eta \bm{F}_1 + \eta^2 \bm{F}_2 + \mathcal{O}(\eta^3)$. 
Assuming $\bm{F}_0$ is diagonal, we may write it without loss of generality as
\begin{align*}
\bm{F}_0 = \operatorname{diag}\Bigl(\underbrace{f_0,\dots,f_0}_{r},\,f_1,\dots,f_{n-r}\Bigr),
\quad f_j \neq f_0 \quad (j=1,\dots,n-r),
\quad 1 \leq r \leq n,
\end{align*}
where $f_0$ has algebraic multiplicity $r$. 
The eigenvalues of $\bm{F}_\eta$ associated with the $f_0$-group admit the asymptotic approximation
up to $\mathcal{O}(\eta^2)$ by the eigenvalues of the corresponding effective Hamiltonian 
\begin{align}\label{f49}
	\bm{P}f_0\bm{P}+\eta \bm{P}\bm{F}_1\bm{P}+\eta^2\bm{P}(\bm{F}_1\bm{G}\bm{F}_1+\bm{F}_2)\bm{P}+\mathcal{O}(\eta^3),
\end{align}
where $\bm{P} = \operatorname{diag}(\underbrace{1,\dots,1}_{r},0,\dots,0)$ is the orthogonal projection onto the $f_0$-eigenspace, and operator $\bm{G}$ is given by
\begin{align*}
\bm{G} = \operatorname{diag}\left(0,\dots,0,\,(f_0-f_1)^{-1},\dots,(f_0-f_{n-r})^{-1}\right).
\end{align*}
For a first-order asymptotic EP, the truncated matrix $\bm{F}_0 + \eta \bm{F}_1$ must be non-diagonalizable.
This occurs when at least one eigenvalue has algebraic multiplicity $\geq 2$ but geometric multiplicity strictly less than its algebraic multiplicity

From \eqref{f49}, we conclude that $\bm{F}_0$ must have at least one eigenvalue with algebraic multiplicity $\geq 2$. 
Without loss of generality, assume $(\bm{F}_0)_{ii} = (\bm{F}_0)_{jj}$ for $i \neq j$ is a double eigenvalue. 
The eigenvalues of $\bm{F}_0 + \eta \bm{F}_1$ are determined by the effective Hamiltonian
\begin{align*}
\begin{pmatrix}
(\bm{F}_0)_{ii} & 0 \\
0 & (\bm{F}_0)_{jj}
\end{pmatrix}
+ \eta
\begin{pmatrix}
\left(\bm{A}_1^{(0)}\right)_{ii} & \left(\bm{A}_1^{(n_i-n_j)}\right)_{ij} \\
\left(\bm{A}_1^{(n_j-n_i)}\right)_{ji} & \left(\bm{A}_1^{(0)}\right)_{jj}
\end{pmatrix}.
\end{align*}
The corresponding eigenvalues are
\begin{align*}
\left(\bm{F}_{0}\right)_{i i}+\eta\left(\frac{\left(\bm{A}_{1}^{(0)}\right)_{i i}+\left(\bm{A}_{1}^{(0)}\right)_{j j}}{2} \pm \sqrt{\left(\frac{\left(\bm{A}_{1}^{(0)}\right)_{i i}-\left(\bm{A}_{1}^{(0)}\right)_{j j}}{2}\right)^{2}+\left(\bm{A}_{1}^{\left(n_{i}-n_{j}\right)}\right)_{i j}\left(\bm{A}_{1}^{\left(n_{j}-n_{i}\right)}\right)_{j i}}\right).
\end{align*}
When $\bm{A}_1^{(0)} = \bm{0}$, the matrix $\bm{F}_0 + \eta \bm{F}_1$ is non-diagonalizable if and only if the following holds (see \cite{MR4389752}):
\begin{itemize}
    \item  $\bm{F}_0$ has a repeated eigenvalue ($(\bm{F}_0)_{ii} = (\bm{F}_0)_{jj}$ for $i \neq j$),
    \item Exactly one of the off-diagonal elements $\left(\bm{A}_1^{(n_i-n_j)}\right)_{ij}$ or $\left(\bm{A}_1^{(n_j-n_i)}\right)_{ji}$ vanishes.
\end{itemize}

Building on the preceding analysis, we now construct asymptotic EPs for the wave equation \eqref{f29} derived in Section~\ref{sec4}. For this purpose, we define the modulated density functions
\begin{align*}
\widetilde{\bm{\rho}}_j(\eta, t) &= \operatorname{diag}\left(\widetilde{\rho}_{j1}(\eta, t), \widetilde{\rho}_{j2}(\eta, t), \cdots, \widetilde{\rho}_{jd}(\eta, t)\right), \quad j = 1, \cdots, N, \\
\frac{1}{\widetilde{\rho}_{js}(\eta, t)} &= 1 + \eta \xi_{js}(t), \quad s = 1, \cdots, d,
\end{align*}
where each modulation function $\xi_{ji}(t)$ is $T$-periodic with finite Fourier series. 

The coefficient matrix for \eqref{f29} takes the form $\bm{B}_\eta(t) = \bm{B}_0 + \eta \bm{B}_1(t)$. To simplify notation while preserving parametric dependence, we express the first-order system as
\begin{align}\label{f43}
\frac{\mathrm{d} \bm{y}(t)}{\mathrm{d}t} = \bm{A}_\eta(t) \bm{y}(t), \quad 
\bm{A}_\eta(t) = \bm{A}_0 + \eta \bm{A}_1(t),
\end{align}
with block matrices
\begin{align*}
\bm{A}_0 = \begin{pmatrix}
\bm{0}_{dN} & \bm{I}_{dN} \\
\bm{B}_0 & \bm{0}_{dN}
\end{pmatrix}, \quad
\bm{A}_1(t) = \begin{pmatrix}
\bm{0}_{dN} & \bm{0}_{dN} \\
\bm{B}_1(t) & \bm{0}_{dN}
\end{pmatrix}.	
\end{align*}

Assume $\bm{A}_0$ is diagonalizable with eigenvalues $\lambda_i$ ($i = 1, \dots, 2dN$) listed by algebraic multiplicity. There exists an invertible matrix $\bm{V}$ such that
\begin{align*}
\bm{V}^{-1} \bm{A}_0 \bm{V} = \operatorname{diag}(\lambda_1, \lambda_2, \dots, \lambda_{2dN}).
\end{align*}
Let $\bm{\vartheta}_i$ denote the eigenvector corresponding to $\lambda_i$, with $\bm{\vartheta}'_i = (\vartheta_{1i}, \dots, \vartheta_{(dN)i})^\top$ representing its first $dN$ components. These satisfy
\begin{align*}
(\bm{B}_0 - \lambda_i^2 \bm{I}_{dN}) \bm{\vartheta}'_i = \bm{0}, \quad
\vartheta_{(dN+j)i} = \lambda_i \vartheta_{ji},  \quad i,j=1,\dots,dN,
\end{align*} 
and $\det(\lambda_i^2 \bm{I}_{dN} - \bm{B}_0) = 0$. 

Let $\bm{S}$ diagonalize $\bm{B}_0$ such that
\begin{align*}
\bm{S}^{-1} \bm{B}_0 \bm{S} = \operatorname{diag}(\lambda_1^2, \lambda_2^2, \dots, \lambda_{dN}^2).
\end{align*}
Through eigenvalue reordering and similarity transformation, we obtain the structured decomposition
\begin{align*}
\bm{V}^{-1} \bm{A}_0 \bm{V} = \operatorname{diag}(\lambda_1, \dots, \lambda_{dN}, -\lambda_1, \dots, -\lambda_{dN}),
\end{align*}
with eigenmatrix
\begin{align*}
\bm{V} = \begin{pmatrix}
\bm{S} & \bm{S} \\
\bm{S} \bm{\Phi} & -\bm{S} \bm{\Phi}
\end{pmatrix}, \quad
\bm{\Phi} = \operatorname{diag}(\lambda_1, \lambda_2, \dots, \lambda_{dN}).
\end{align*}
Crucially, $\bm{A}_0$ is diagonalizable if and only if $\bm{B}_0$ is diagonalizable and $\det(\bm{B}_0) \neq 0$.

 Assuming $\lambda_i \neq 0$ for all $i = 1, 2, \dots, dN$ (ensuring $\bm{\Phi}$ is invertible), the block matrix inversion formula yields
\begin{align*}
\bm{V}^{-1} = \frac{1}{2} \begin{pmatrix}
\bm{S}^{-1} & (\bm{S}\bm{\Phi})^{-1} \\
\bm{S}^{-1} & -(\bm{S}\bm{\Phi})^{-1}
\end{pmatrix}.
\end{align*}
The similarity transformation of $\bm{A}_1(t)$ is then given by
\begin{equation}\label{f38}
\bm{V}^{-1}\bm{A}_1(t)\bm{V} = \frac{1}{2} \begin{pmatrix}
\bm{\Phi}^{-1}\bm{S}^{-1}\bm{B}_1(t)\bm{S} & \bm{\Phi}^{-1}\bm{S}^{-1}\bm{B}_1(t)\bm{S} \\
-\bm{\Phi}^{-1}\bm{S}^{-1}\bm{B}_1(t)\bm{S} & -\bm{\Phi}^{-1}\bm{S}^{-1}\bm{B}_1(t)\bm{S}
\end{pmatrix}.
\end{equation}

\begin{assu}\label{assump:system_f43}
We impose the following assumptions on system \eqref{f43}:
\begin{enumerate}[label=(\arabic*)]
  \item Each modulation function $\xi_{ji}(t)$ is $T$-periodic and admits a finite Fourier series with vanishing zero-frequency component.
  \item The matrix $\bm{B}_0$ is diagonalizable and invertible ($\det(\bm{B}_0) \neq 0$), which implies that $\bm{A}_0$ is diagonalizable. 
\end{enumerate}
\end{assu}

Under Assumption~\ref{assump:system_f43}, we propose a construction scheme for asymptotic EPs in system \eqref{f43}.

\begin{theo}\label{th3}
Let Assumption~\ref{assump:system_f43} hold. System \eqref{f43} exhibits a first-order asymptotic EP if and only if:
\begin{enumerate}[label=(\arabic*)]
  \item There exist distinct eigenvalues $\lambda_i^2, \lambda_j^2$ ($i \neq j$) of $\bm{B}_0$ satisfying
        \begin{align*}
          \lambda_i \equiv \lambda_j \pmod{\mathrm{i}\Omega}
        \end{align*}
        for modulation frequency $\Omega = 2\pi/T$;
        
  \item Letting $\bm{S}$ diagonalize $\bm{B}_0$, the corresponding matrix elements of the first-order perturbation in \eqref{f38} satisfy
        \begin{align*}
          \left[ (\bm{S}^{-1}\bm{B}_1\bm{S})^{(n_i - n_j)} \right]_{ij} = 0
          \quad\text{or}\quad
          \left[ (\bm{S}^{-1}\bm{B}_1\bm{S})^{(n_j - n_i)} \right]_{ji} = 0,
        \end{align*}
        where $n_i$, $n_j$ denote the folding numbers of $\lambda_i$, $\lambda_j$ respectively.
\end{enumerate}
\end{theo}

To apply Theorem~\ref{th3} to the elastic wave equation, we consider a three-dimensional configuration with a single resonator. Defining
\[
c_{s's} := \frac{1}{\rho_{s'}} C_{11}^{\alpha, s, s'} \quad (s, s' = 1,2,3),
\]
where $c_{ss}$ is real-valued for each $s$, we obtain from \eqref{f29}:
\begin{align}\label{f30}
	\frac{{\rm d} \bm{y}(t)}{{\rm d}t} = \bm{A}_\eta(t) \bm{y}(t), \quad \bm{A}_\eta(t) = \bm{A}_0 + \eta \bm{A}_1(t),
\end{align}
with
\begin{align*}
&\bm{A}_0 = \begin{pmatrix}
		\bm{0} & \bm{I}_{3} \\
		\bm{B}_0 & \bm{0}
	\end{pmatrix}, \quad
\bm{A}_1(t) = \begin{pmatrix}
		\bm{0} & \bm{0} \\
		\bm{B}_1(t) & \bm{0}
	\end{pmatrix}, \\
&\bm{B}_0 = \frac{\varepsilon}{|D|}
\begin{pmatrix}
	c_{11} & c_{12} & c_{13} \\
	\overline{c_{12}} & c_{22} & c_{23} \\
	\overline{c_{13}} & \overline{c_{23}} & c_{33}
\end{pmatrix}, \quad
\bm{B}_1(t) = \frac{\varepsilon}{|D|}
\begin{pmatrix}
	\xi_1(t) c_{11} & \xi_1(t) c_{12} & \xi_1(t) c_{13} \\
	\xi_2(t) \overline{c_{12}} & \xi_2(t) c_{22} & \xi_2(t) c_{23} \\
	\xi_3(t) \overline{c_{13}} & \xi_3(t) \overline{c_{23}} & \xi_3(t) c_{33}
\end{pmatrix}.
\end{align*}

The characteristic polynomial of $\bm{B}_0$ is given by:
\begin{equation*}
\begin{aligned}
\det(\bm{B}_0 - \lambda^2 \bm{I}_3)
&= -(\lambda^2)^3 + \frac{\varepsilon}{|D|}(c_{11} + c_{22} + c_{33})(\lambda^2)^2 \\
&\quad + \frac{\varepsilon^2}{|D|^2} \left( -c_{11} c_{22} - c_{11} c_{33} - c_{22} c_{33} + |c_{12}|^2 + |c_{13}|^2 + |c_{23}|^2 \right) \lambda^2 \\
&\quad + \frac{\varepsilon^3}{|D|^3} \big( c_{11} c_{22} c_{33} - c_{11}|c_{23}|^2 - c_{22} |c_{13}|^2 - c_{33}|c_{12}|^2 \\
&\qquad + c_{12} \overline{c_{13}} c_{23} + \overline{c_{12}} c_{13} \overline{c_{23}} \big).
\end{aligned}
\end{equation*}

Denoting the discriminant of this cubic polynomial by $\Delta_3$, we assume $\Delta_3 \neq 0$ (ensuring three distinct eigenvalues), $c_{12} \neq 0$, and that for each eigenvalue $\lambda_i^2$,
\begin{align}\label{f52}
\left( \frac{\varepsilon}{|D|} c_{11} - \lambda_i^2 ,\, \frac{\varepsilon}{|D|} c_{12} ,\, \frac{\varepsilon}{|D|} c_{13} \right) \not\parallel \left( \frac{\varepsilon}{|D|} \overline{c_{12}} ,\, \frac{\varepsilon}{|D|} c_{22} - \lambda_i^2 ,\, \frac{\varepsilon}{|D|} c_{23} \right),
\end{align}
where $\not\parallel$ denotes non-proportionality. The eigenvector corresponding to $\lambda_i^2$ is then 
\begin{align*}
 \bm{ \vartheta}_i = \begin{pmatrix}
	\dfrac{\varepsilon}{|D|^2} \left( -c_{13} c_{22} \varepsilon + c_{12} c_{23} \varepsilon + c_{13} |D| \lambda_i^2 \right) \\[10pt]
	\dfrac{\varepsilon}{|D|^2} \left( -c_{11} c_{23} \varepsilon + c_{23} |D| \lambda_i^2 +\overline{c_{12}} c_{13}  \varepsilon \right) \\[10pt]
	\dfrac{1}{|D|^2} \left( (c_{11} \varepsilon - |D| \lambda_i^2)(c_{22} \varepsilon - |D| \lambda_i^2) - |c_{12}|^2 \varepsilon^2 \right)
\end{pmatrix}.
\end{align*}
The eigenmatrix $\bm{S} =\begin{pmatrix}
	\bm{\vartheta}_1, \bm{\vartheta}_2, \bm{\vartheta}_3 \end{pmatrix}$ consequently satisfies
\begin{align}\label{f53}
\det(\bm{S}) = -\frac{\varepsilon^3 }{|D|^3} (\lambda_1^2 - \lambda_2^2) (\lambda_1^2 - \lambda_3^2) (\lambda_2^2 - \lambda_3^2) \left(c_{12} c_{23}^2 - \overline{c_{12}} c_{13}^2  \right) \neq 0.
\end{align}

Next, we construct specific examples of first-order asymptotic EPs based on the preceding analysis. Due to the high dimensionality of the parameter space, we fix select parameters judiciously. Following Theorem~\ref{th3}, we simplify by setting
\begin{align}\label{f51}
\xi_2(t) = \xi_3(t) = 0, \quad c_{11} = c_{22} = c_{33},
\end{align}
and fix $i=1$, $j=2$ as specified in the theorem. This yields the matrix elements:
\begin{align*}
(\bm{S}^{-1}\bm{B}_1\bm{S})_{12}(t) &= \frac{
  \xi_{1}(t) \lambda_2^2 f(\lambda_2^2) g(\lambda_1^2) }{
  |D|^2 \varepsilon (\lambda_1^2 - \lambda_2^2) (\lambda_1^2 - \lambda_3^2) 
  \left( c_{12} c_{23}^2 - \overline{c_{12}} c_{13}^2 \right)
}, \\
(\bm{S}^{-1}\bm{B}_1\bm{S})_{21}(t) &= \frac{
  \xi_{1}(t) \lambda_1^2 f(\lambda_1^2) g(\lambda_2^2) }{
  |D|^2 \varepsilon (\lambda_1^2 - \lambda_2^2) (\lambda_2^2 - \lambda_3^2) 
  \left( c_{12} c_{23}^2 - \overline{c_{12}} c_{13}^2 \right)
},
\end{align*}
where
\begin{align*}
f(\lambda) &= c_{13} |D| \lambda + \varepsilon (c_{12} c_{23} - c_{11} c_{13}), \\
g(\lambda) &= c_{23} |D|^2 \lambda^2 +\left( - \overline{c_{12}} c_{13} - 2c_{11} c_{23} \right) |D|\varepsilon \lambda  \\
           &\quad + \varepsilon^2 \left( c_{11}^2 c_{23} + c_{11} \overline{c_{12}} c_{13} - c_{23} (|c_{13}|^2 + |c_{23}|^2) \right).
\end{align*}

Under the above parameter constraints, we derive:
\begin{align}
\det(\bm{B}_0) &= \frac{\varepsilon^3}{|D|^3} \left( c_{11}^3 - c_{11} (|c_{12}|^2 + |c_{13}|^2 + |c_{23}|^2) + c_{12}\overline{c_{13}} c_{23}  + \overline{c_{12}} c_{13} \overline{c_{23}} \right) \neq 0, \label{f54} \\
\begin{split}
\det(\bm{B}_0 - \lambda^2 \bm{I}_3) &= -(\lambda^2)^3 + \frac{3c_{11}\varepsilon}{|D|}(\lambda^2)^2 \\
&\quad + \frac{\varepsilon^2}{|D|^2} \left( -3c_{11}^2 + |c_{12}|^2 + |c_{13}|^2 + |c_{23}|^2 \right) \lambda^2 \\
&\quad + \frac{\varepsilon^3}{|D|^3} \big( c_{11}^3 - c_{11} (|c_{12}|^2 + |c_{13}|^2 + |c_{23}|^2) \\
&\qquad + c_{12} \overline{c_{13}} c_{23} + \overline{c_{12}} c_{13} \overline{c_{23}} \big), \label{f39}
\end{split}
\end{align}
with the discriminant of \eqref{f39} assumed negative:
\begin{align} 
\Delta_3 &= \frac{\varepsilon^6}{|D|^6} \left( \frac{|c_{12} \overline{c_{13}} c_{23} + \overline{c_{12}} c_{13} \overline{c_{23}}|^2}{4} - \frac{(|c_{12}|^2 + |c_{13}|^2 + |c_{23}|^2)^3}{27} \right) < 0. \label{f42}
\end{align}
Vieta's relations yield:
\begin{align}
\lambda_1^2 + \lambda_2^2 + \lambda_3^2 &= \frac{3c_{11}\varepsilon}{|D|}, \label{f40} \\
\lambda_1^2 \lambda_2^2 \lambda_3^2 &= \frac{\varepsilon^3}{|D|^3} \left( c_{11}^3 - c_{11}(|c_{12}|^2 + |c_{13}|^2 + |c_{23}|^2) + c_{12} \overline{c_{13}} c_{23} + \overline{c_{12}} c_{13} \overline{c_{23}} \right). \label{f41}
\end{align}
Thus, knowledge of one eigenvalue determines the others through \eqref{f40} and \eqref{f41}. 

To satisfy condition (2) of Theorem~\ref{th3} and leveraging the symmetry in $(\bm{S}^{-1}\bm{B}_1\bm{S})_{12}(t)$ and $(\bm{S}^{-1}\bm{B}_1\bm{S})_{21}(t)$, it suffices to consider
\begin{align*}
\left( \bm{S}^{-1}\bm{B}_1\bm{S} \right)^{(n_1-n_2)}_{12} = 0 \quad \text{while} \quad \left( \bm{S}^{-1}\bm{B}_1\bm{S} \right)^{(n_2-n_1)}_{21} \neq 0.
\end{align*}
The condition $\left( \bm{S}^{-1}\bm{B}_1\bm{S} \right)^{(n_1-n_2)}_{12} = 0$ reduces to two cases:
\begin{align*}
f(\lambda_2^2) = 0 \quad \text{or} \quad g(\lambda_1^2) = 0.
\end{align*}

Building upon the preceding theoretical framework, we now develop a constructive method for first-order asymptotic EPs in system~\eqref{f30}. The procedure bifurcates into two distinct cases:

\noindent
\textbf{Case 1: Construction via the equation $f(x) = 0$}
\begin{enumerate}[label=(\arabic*), leftmargin=*]
    \item \textbf{Eigenvalue selection}
    \begin{itemize}
        \item Solve $f(x) = 0$ and select a non-zero root $x$ satisfying $g(x) \neq 0$
        \item Designate $\lambda_2^2 = x$ as an eigenvalue of $\bm{B}_0$
        \item Note: $\bm{B}_0$ possesses three non-zero eigenvalues $\{\lambda_k^2\}_{k=1}^3$ consisting of one real and two complex-conjugate values
    \end{itemize}
    
    \item \textbf{Remaining eigenvalue determination}
    \begin{itemize}
        \item Compute the other eigenvalues $y$ and $z$ using Vieta's relations \eqref{f40} and \eqref{f41} with $\lambda_2^2 = x$
        \item Assign $\lambda_1^2$ to either $y$ or $z$
    \end{itemize}
    
    \item \textbf{Modulation constraints}
    \begin{itemize}
        \item Apply Theorem~\ref{th3}(1) to $\lambda_1^2$ and $\lambda_2^2$ to establish the relationship between modulation frequency $\Omega$ and system parameters
        \item Enforce Condition (2) of Theorem~\ref{th3}: require non-vanishing of the $(n_2 - n_1)$-th Fourier coefficient of $\xi_1(t)$
    \end{itemize}
\end{enumerate}
\textit{Implementation guidance:} See Examples 1--2 for detailed demonstrations of root selection and eigenvalue assignment procedures.

\noindent
\textbf{Case 2: Construction via the equation $g(x) = 0$}
\begin{enumerate}[label=(\arabic*), leftmargin=*]
    \item \textbf{Root analysis and eigenvalue selection}
    \begin{itemize}
        \item Solve $g(x) = 0$ to obtain roots $\{x_1, x_2\}$
        \item Branch based on root configuration:
        \begin{itemize}
            \item \textit{Complex conjugate eigenvalues}: Set $\lambda_1^2 = x_1$ (without loss of generality) and proceed to step (2)
            
            \item \textit{Non-conjugate eigenvalues}: For each distinct root $x_i \in \{x_1, x_2\}$, set $\lambda_1^2 = x_i$ and execute Case 1 steps (1)--(2)
        \end{itemize}
    \end{itemize}
    
    \item \textbf{Eigenvalue assignment for conjugate case}
    \begin{itemize}
        \item Compute remaining eigenvalues $y, z$ using Vieta's relations \eqref{f40} and \eqref{f41} with $\lambda_1^2 = x_1$
        \item Assign $\lambda_2^2$ to the real eigenvalue in $\{y, z\}$, excluding $x_2$ to ensure $g(\lambda_2^2) \neq 0$
        \item \textit{Rationale}: Avoids conflict with $x_2 = \overline{x_1}$ while satisfying the non-degeneracy condition
    \end{itemize}
    
    \item \textbf{Modulation constraints}
    \begin{itemize}
        \item Apply Theorem~\ref{th3}(1) to $\lambda_1^2$ and $\lambda_2^2$ to establish the $\Omega$-parameter relationship
        \item Enforce Condition (2) of Theorem~\ref{th3}: Require non-vanishing $(n_2 - n_1)$-th Fourier coefficient of $\xi_1(t)$
    \end{itemize}
\end{enumerate}
\textit{Implementation guidance:} Examples 3--4 provide case-specific procedures and eigenvalue selection criteria.

Leveraging the developed methodology, we present four concrete examples of first-order asymptotic EP constructions for system \eqref{f29}. The following examples utilize the standard formula for computing square roots of complex numbers with non-zero real and imaginary components:

\begin{rema}
For a complex number $A + \mathrm{i}B$ with $A \neq 0$ and $B \neq 0$, its square roots are given by
\begin{align*}
    \pm \left( \sqrt{ \dfrac{ \sqrt{A^2 + B^2} + A }{2} } + \mathrm{i} \cdot \sgn(B) \sqrt{ \dfrac{ \sqrt{A^2 + B^2} - A }{2} } \right).
\end{align*}
\end{rema}

Building on the preceding analysis, we assume all parameters satisfy conditions \eqref{f52}, \eqref{f53}, \eqref{f51}, \eqref{f54}, and \eqref{f42}. We now detail the EP construction procedure.

First, consider the case $f(\lambda_2^2) = 0$. Following Case~1-(1), this yields:
\begin{align*}
\lambda_2^2 = \frac{\varepsilon (c_{11}c_{13} - c_{12}c_{23})}{c_{13} |D|} \neq 0, \quad c_{13} \neq 0,
\end{align*}
which must simultaneously satisfy the characteristic equation:
\begin{align*}
\det(\bm{B}_0 - \lambda_2^2 \bm{I}_3) = \frac{\varepsilon^3}{c_{13}^3 |D|^3} \left( c_{12}^3 c_{23}^3 - c_{12} c_{13}^2 c_{23} ( |c_{12}|^2 + |c_{23}|^2 ) + \overline{c_{12}} c_{13}^4 \overline{c_{23}} \right) = 0.
\end{align*}

Since $\lambda_1^2 \neq \lambda_2^2$, we have $f(\lambda_1^2) \neq 0$. To ensure non-vanishing of the perturbation matrix element
\[
\left( (\bm{S}^{-1}\bm{B}_1(t)\bm{S})^{(n_2-n_1)} \right)_{21} \neq 0,
\]
we impose the dual conditions:
\begin{align*}
\xi_1^{(n_2-n_1)} \neq 0 \quad \text{and} \quad g(\lambda_2^2) \neq 0,
\end{align*}
where the latter expands to:
\begin{align*}
c_{23} \varepsilon^2 \left( c_{12}^2 c_{23}^2 + c_{13}^2 (|c_{12}|^2 - |c_{13}|^2 - |c_{23}|^2) \right) \neq 0.
\end{align*}
The remaining eigenvalues of $\bm{B}_0$ are then computed using discriminant condition \eqref{f42} and Vieta's relations \eqref{f40}, \eqref{f41}.

We now address Case 1-(2) and Case 1-(3) through two distinct examples categorized by the nature of $\lambda_2^2$:

\begin{enumerate}[leftmargin=*, label=\textbf{Example \arabic*}:]
    \item \textbf{Real eigenvalue configuration} ($\lambda_2^2 \in \mathbb{R}$ and $\displaystyle\frac{c_{11}c_{13} - c_{12}c_{23}}{c_{13}} > 0$):
    \begin{align*}
        \lambda_2^2 &= \frac{\varepsilon (c_{11}c_{13} - c_{12}c_{23})}{|D| c_{13}} \\
        \lambda_{2\pm} &= \pm \sqrt{ \lambda_2^2 }
    \end{align*}
    The remaining eigenvalues form a complex-conjugate pair:
    \begin{align*}
        \alpha^2 &= a + \mathrm{i} b, \quad 
        \beta^2 = a - \mathrm{i} b \quad (b > 0)
    \end{align*}
    with components:
    \begin{align*}
        a &= \dfrac{\varepsilon}{|D|} \cdot \dfrac{2c_{11}c_{13} + c_{12}c_{23}}{2c_{13}}, \\
        a^2 + b^2 &= \dfrac{\varepsilon^2}{|D|^2} \cdot \dfrac{ c_{13} \left( c_{11}^3 - c_{11} (|c_{12}|^2 + |c_{13}|^2 + |c_{23}|^2) + c_{12} \overline{c_{13}} c_{23} + \overline{c_{12}} c_{13} \overline{c_{23}} \right) }{c_{11} c_{13} - c_{12} c_{23}}.
    \end{align*}
    Their square roots are:
    \begin{align*}
        \alpha_{\pm} &= \pm \left( \sqrt{ \dfrac{ \sqrt{a^2 + b^2} + a }{2} } + \mathrm{i} \sqrt{ \dfrac{ \sqrt{a^2 + b^2} - a }{2} } \right), \\
        \beta_{\pm} &= \pm \left( \sqrt{ \dfrac{ \sqrt{a^2 + b^2} + a }{2} } - \mathrm{i} \sqrt{ \dfrac{ \sqrt{a^2 + b^2} - a }{2} } \right).
    \end{align*}
   By Theorem~\ref{th3}, the congruence conditions 
$\lambda_{2+} \equiv \lambda_1 \pmod{\mathrm{i}\Omega}$ for $\lambda_1 \in \{ \alpha_{+}, \beta_{+} \}$ 
or $\lambda_{2-} \equiv \lambda_1 \pmod{\mathrm{i}\Omega}$ for $\lambda_1 \in \{ \alpha_{-}, \beta_{-} \}$ 
both necessitate identical real parts. This requirement yields:
\begin{align*}
\dfrac{(2 c_{11} c_{13} - 5 c_{12} c_{23})^2}{4 c_{13}} = \dfrac{ c_{11}^3 - c_{11} (|c_{12}|^2 + |c_{13}|^2 + |c_{23}|^2) + c_{12} \overline{c_{13}} c_{23} + \overline{c_{12}} c_{13} \overline{c_{23}} }{c_{11} c_{13} - c_{12} c_{23}}.
\end{align*}
    The modulation frequency and Fourier constraints are:
    \begin{align*}
        \Omega &= \dfrac{1}{n} \sqrt{ \dfrac{ \sqrt{a^2 + b^2} - a }{2} }, \quad n \in \mathbb{N} \setminus \{0\}, \\
        \xi_1^{(-n)} \neq 0 \quad &\text{if} \quad \lambda_1 \in \{ \alpha_{+}, \beta_{-} \}, \\
        \xi_1^{(n)} \neq 0 \quad &\text{if} \quad \lambda_1 \in \{ \alpha_{-}, \beta_{+} \}.
    \end{align*}

    \item \textbf{Complex eigenvalue configuration} ($\lambda_2^2  \in \mathbb{C}$):
    The eigenvalues comprise $\lambda_2^2$, its conjugate $\tau^2$, and real $\gamma^2$:
    \begin{align*}
        \tau^2 &= \frac{\varepsilon(c_{11}\overline{c_{13}} - \overline{c_{12}}\overline{c_{23}})}{|D|\overline{c_{13}}}, \\
        \gamma^2 &= \frac{\varepsilon(c_{11}|c_{13}|^2 + c_{12}\overline{c_{13}}c_{23} + \overline{c_{12}}c_{13}\overline{c_{23}})}{|D||c_{13}|^2}.
    \end{align*}
    The real part and the squared modulus of $\lambda_2^2$ satisfy:
    \begin{align*}
        \Re(\lambda_2^2) &= \Re(\tau^2) = \dfrac{\varepsilon(2c_{11}|c_{13}|^2 - c_{12}\overline{c_{13}}c_{23} - \overline{c_{12}}c_{13}\overline{c_{23}})}{2|D||c_{13}|^2}, \\
        |\lambda_2^2|^2 &=(\Re(\lambda_2^2))^2+(\Im(\lambda_2^2))^2= \dfrac{\varepsilon^2|c_{11}c_{13} - c_{12}c_{23}|^2}{|D|^2|c_{13}|^2}.
    \end{align*}
    Square roots are:
    \begin{align*}
        \lambda_{2\pm} &= \pm \left( \sqrt{ \dfrac{ |\lambda_2^2| + \Re(\lambda_2^2) }{2} } + \mathrm{i} \cdot \sgn(\Im(\lambda_2^2)) \sqrt{ \dfrac{ |\lambda_2^2| - \Re(\lambda_2^2) }{2} } \right), \\
        \tau_{\pm} &= \pm \left( \sqrt{ \dfrac{ |\lambda_2^2| + \Re(\lambda_2^2) }{2} } - \mathrm{i} \cdot \sgn(\Im(\lambda_2^2)) \sqrt{ \dfrac{ |\lambda_2^2| - \Re(\lambda_2^2) }{2} } \right).
    \end{align*}
    
    \begin{enumerate}[label=(\arabic*), leftmargin=*]
    \item For $\lambda_{2+} \equiv \tau_{+} \pmod{\mathrm{i}\Omega}$ or $\lambda_{2-} \equiv \tau_{-} \pmod{\mathrm{i}\Omega}$, the modulation frequency satisfies
    \begin{align*}
        \Omega = \frac{1}{n} \sqrt{ \frac{ |\lambda_2^2| - \Re(\lambda_2^2) }{2} }, \quad n \in \mathbb{N} \setminus \{0\},
    \end{align*}
    with corresponding Fourier coefficient constraints:
    \begin{align*}
        \xi_1^{(2n \cdot \sgn(\Im(\lambda_2^2)))} \neq 0 \quad &\text{if} \quad \lambda_1 = \tau_{+}, \\
        \xi_1^{(-2n \cdot \sgn(\Im(\lambda_2^2)))} \neq 0 \quad &\text{if} \quad \lambda_1 = \tau_{-}.
    \end{align*}
    
    \item When $c_{11}|c_{13}|^2 + c_{12}\overline{c_{13}}c_{23} + \overline{c_{12}}c_{13}\overline{c_{23}} > 0$, 
    \begin{align*}
        \gamma_{\pm} = \pm \sqrt{\dfrac{\varepsilon}{|D|}} \cdot \dfrac{\sqrt{c_{11}|c_{13}|^2 + c_{12}\overline{c_{13}}c_{23} + \overline{c_{12}}c_{13}\overline{c_{23}}}}{|c_{13}|}.
    \end{align*}
    For $\lambda_{2+} \equiv \gamma_{+} \pmod{\mathrm{i}\Omega}$ or $\lambda_{2-} \equiv \gamma_{-} \pmod{\mathrm{i}\Omega}$, the real-part coincidence condition implies
    \begin{align*}
        2|c_{13}| \cdot |c_{11} c_{13} - c_{12} c_{23}| = 2c_{11}|c_{13}|^2 + 5 \left( c_{12} \overline{c_{13}} c_{23} + \overline{c_{12}} c_{13} \overline{c_{23}} \right),
    \end{align*}
    with modulation frequency
    \begin{align*}
        \Omega = \frac{1}{n} \sqrt{ \frac{ |\lambda_2^2| - \Re(\lambda_2^2) }{2} }, \quad n \in \mathbb{N} \setminus \{0\},
    \end{align*}
    and constraints:
    \begin{align*}
        \xi_1^{(n \cdot \sgn(\Im(\lambda_2^2)))} \neq 0 \quad &\text{if} \quad \lambda_1 = \gamma_{+}, \\
        \xi_1^{(-n \cdot \sgn(\Im(\lambda_2^2)))} \neq 0 \quad &\text{if} \quad \lambda_1 = \gamma_{-}.
    \end{align*}
\end{enumerate}

\end{enumerate}

\begin{rema}
The preceding examples exclude two cases due to inherent contradictions:
\begin{enumerate}[label=(\arabic*), leftmargin=*]
    \item For real $\lambda_2^2 = \dfrac{\varepsilon (c_{11}c_{13} - c_{12}c_{23})}{|D| c_{13}}$ with $\dfrac{c_{11}c_{13} - c_{12}c_{23}}{c_{13}} < 0$:
    \begin{align*}
        \lambda_{2\pm} = \pm \mathrm{i} \sqrt{ -\dfrac{\varepsilon (c_{11}c_{13} - c_{12}c_{23})}{|D| c_{13}} }.
    \end{align*}
    The remaining eigenvalues $\alpha^2$, $\beta^2$ form a complex-conjugate pair as in Example 1. The congruence conditions
    \begin{align*}
        \lambda_{2+} \equiv \lambda_1 \pmod{\mathrm{i}\Omega} \quad &(\lambda_1 \in \{\alpha_{+}, \beta_{+}\}) \\
        \text{or} \quad \lambda_{2-} \equiv \lambda_1 \pmod{\mathrm{i}\Omega} \quad &(\lambda_1 \in \{\alpha_{-}, \beta_{-}\})
    \end{align*}
    require $\Im(\alpha^2) = \Im(\beta^2) = 0$, contradicting the conjugate-pair property $\Im(\alpha^2) = -\Im(\beta^2) \neq 0$.
    
    \item For complex $\lambda_2^2 = \dfrac{\varepsilon (c_{11}c_{13} - c_{12}c_{23})}{|D| c_{13}}$ with $c_{11}|c_{13}|^2 + c_{12}\overline{c_{13}}c_{23} + \overline{c_{12}}c_{13}\overline{c_{23}} < 0$:
    \begin{align*}
        \gamma_{\pm} = \pm \mathrm{i} \sqrt{\dfrac{\varepsilon}{|D|}} \dfrac{\sqrt{ -c_{11}|c_{13}|^2 - c_{12}\overline{c_{13}}c_{23} - \overline{c_{12}}c_{13}\overline{c_{23}} }}{|c_{13}|}.
    \end{align*}
    The conditions $\lambda_{2\pm} \equiv \gamma_{\pm} \pmod{\mathrm{i}\Omega}$ require $\Re(\lambda_2^2) = \Re(\gamma_{\pm}) = 0$, contradicting the initial complexity assumption $\Im(\lambda_2^2) \neq 0$.
\end{enumerate}
\end{rema}

Based on Case 2, we impose the conditions:
\begin{align*}
\left( \left( \bm{S}^{-1} \bm{B}_1 \bm{S} \right)^{(n_1 - n_2)} \right)_{12} = 0
\quad \text{and} \quad
\left( \left( \bm{S}^{-1} \bm{B}_1 \bm{S} \right)^{(n_2 - n_1)} \right)_{21} \neq 0,
\end{align*}
which lead to \( g(\lambda_1^2) = 0 \).

When \( c_{23} = 0 \), solving yields \( \lambda_1^2 = \displaystyle\frac{\varepsilon}{|D|}c_{11} \) and \( f(\lambda_1^2) = 0 \); this case is excluded from further consideration.  

Assuming \( c_{23} \neq 0 \), the equation \( g(\lambda_1^2) = 0 \) has two roots (possibly coincident):
\begin{align*}
\lambda_{1\pm}^2 = \frac{\varepsilon}{|D|}\left(
    \frac{\overline{c_{12}}c_{13} + 2c_{11}c_{23}}{2c_{23}} 
    \pm \frac{\zeta}{2c_{23}}
\right),
\end{align*}
where \( \zeta = \sqrt{\overline{c_{12}}^2c_{13}^2 + 4|c_{13}|^2 c_{23}^2 + 4|c_{23}|^2 c_{23}^2} \). These satisfy \( \xi^{(n_2-n_1)} \neq 0 \), and
\begin{align*}
\det(\bm{B}_0 - \lambda_{1\pm}^2 \bm{I}_3) 
&= \frac{\varepsilon^3}{2 c_{23}^3 |D|^3} \Biggl( 
    -2  \overline{c_{12}}c_{13} c_{23}^2 \left( |c_{13}|^2 + |c_{23}|^2 \right) 
    -\overline{c_{12}}^3 c_{13}^3 \\
    &\quad +  |c_{12}|^2 \overline{c_{12}}c_{13} c_{23}^2 
    + 2 c_{12} \overline{c_{13}} c_{23}^4   
    + 2\overline{c_{12}} c_{13} c_{23}^3 \overline{c_{23}} \\
    &\quad \pm \left(|c_{12}|^2 c_{23}^2 -\overline{c_{12}}^2c_{13}^2 \right) \zeta 
\Biggr) \\
&= 0,
\end{align*}
\begin{align*}
f(\lambda_{1\pm}^2) = \frac{\varepsilon}{2 c_{23}} \left(2 c_{12} c_{23}^2 + \overline{c_{12}} c_{13}^2 \pm c_{13} \zeta\right) \neq 0.
\end{align*}

\noindent\textbf{Note on $\pm$ convention:}
\begin{itemize}
    \item $\lambda_{1+}^2$ consistently uses the upper sign ($+$) in all expressions
    \item $\lambda_{1-}^2$ consistently uses the lower sign ($-$) in all expressions
\end{itemize}
The equations apply \textit{separately} to each root; $\lambda_{1+}^2$ and $\lambda_{1-}^2$ represent distinct solution branches.

Building on the preceding analysis, we address Cases 2-(2) and 2-(3) through two distinct scenarios based on the nature of $\lambda_{1\pm}^2$:

\begin{enumerate}[label=\textbf{Example \arabic*}:, start=3]
\item \textbf{Complex conjugate roots case:} 
When $\lambda_{1+}^2$ and $\lambda_{1-}^2$ form a complex conjugate pair, the parameters must satisfy:
\begin{enumerate}[label=(\arabic*)]
    \item \textit{Equal real parts:} $\Re(\lambda_{1+}^2) = \Re(\lambda_{1-}^2)$ requires
        \begin{equation*}
            \Re(c_{23}) \cdot \Re(\zeta) = \Im(c_{23}) \cdot \Im(\zeta)
        \end{equation*}
    \item \textit{Opposite imaginary parts:} $\Im(\lambda_{1+}^2) = -\Im(\lambda_{1-}^2)$ requires
        \begin{equation*}
           \Im\left( \frac{\overline{c_{12}}c_{13} + 2c_{11}c_{23}}{2c_{23}} \right) = 0
        \end{equation*}
    \item \textit{Non-vanishing imaginary parts:} $\Im(\lambda_{1\pm}^2) \neq 0$ requires
        \begin{equation*}
            \Re(c_{23}) \cdot \Im(\zeta) + \Im(c_{23}) \cdot \Re(\zeta) \neq 0
        \end{equation*}
\end{enumerate}
When either $\lambda_{1+}^{2}$ or $\lambda_{1-}^{2}$ is assigned as an eigenvalue of $\bm{B}_0$, the real eigenvalue $\lambda_{2}^{2}$ must be selected as the complementary eigenvalue. For derivation details, see Example~2-(2).

\item \textbf{Non-conjugate roots case:} 
When the conditions of Example~3 are unsatisfied, $\lambda_{1+}^2$ and $\lambda_{1-}^2$ are not complex conjugates. The analysis proceeds by eigenvalue selection:
\begin{enumerate}[label=(\arabic*)]
  \item For $\lambda_{1+}^2$ as eigenvalue of $\bm{B}_0$:
  \begin{itemize}
    \item Real case: See Example~1
    \item Complex case: See Example~2
  \end{itemize}
  \item For $\lambda_{1-}^2$ as eigenvalue of $\bm{B}_0$:
  \begin{itemize}
    \item Real case: See Example~1
    \item Complex case: See Example~2
  \end{itemize}
\end{enumerate}
\end{enumerate}

\begin{rema}
For a single resonator in two dimensions, Theorem~\ref{th3} cannot construct asymptotic EPs. 
The fundamental obstruction lies in conflicting requirements for the unperturbed system's coefficient matrix $\bm{A}_0$:

\begin{enumerate}[label=(\roman*)]
    \item \textit{Diagonalizability requirement:}  Theorem~\ref{th3} necessitates that $\bm{A}_0$ be diagonalizable. This requires $\det(\bm{B}_0) \neq 0$ for its block submatrix $\bm{B}_0$.
    \item \textit{EP construction requirement:} The asymptotic EP construction in Theorem~\ref{th3}(2) explicitly requires $\det(\bm{B}_0) = 0$ in two dimensions.
\end{enumerate}

These conditions are mutually exclusive in $\mathbb{R}^2$. Consequently, no parameter configuration satisfies both requirements simultaneously, making Theorem~\ref{th3} inapplicable for EP generation in 2D resonators. 
\end{rema}
Detailed derivations appear in Appendix~\ref{app1}.

\appendix  
\section{The Two-Dimensional Case of a Single Resonator}\label{app1}

We consider a single resonator in two dimensions. Denoting $\displaystyle\frac{1}{\rho_{s'}}C_{11}^{\alpha,s,s'}$ as $c_{s's}$ for $s,s' = 1,2$, where $c_{ss}$ are real-valued (by \eqref{f29}), we obtain the system:
\begin{align}\label{f44}
	\frac{{\rm d} \bm{y}(t)}{{\rm d}t} = \left( \bm{A}_0 + \eta \bm{A}_1(t) \right) \bm{y}(t),
\end{align}
with coefficient matrices:
\begin{align*}
\bm{A}_0 &= \begin{pmatrix}
		\bm{0} & \bm{I}_{2} \\
		\bm{B}_0 & \bm{0}
	\end{pmatrix}, 
\quad
\bm{A}_1(t) = \begin{pmatrix}
		\bm{0} & \bm{0} \\
		\bm{B}_1(t) & \bm{0}
	\end{pmatrix}, \\[2ex]
\bm{B}_0 &= \frac{\varepsilon}{|D|}
\begin{pmatrix}
	c_{11} & c_{12} \\
	\overline{c_{12}} & c_{22}
\end{pmatrix}, 
\\
\bm{B}_1(t) &= \frac{\varepsilon}{|D|}
\begin{pmatrix}
	c_{11}\xi_1(t) & c_{12}\xi_1(t) \\
	\overline{c_{12}}\,\xi_2(t) & c_{22}\xi_2(t)
\end{pmatrix}.
\end{align*}

Since the characteristic polynomial of $\bm{B}_0$ is
\[
\det( \bm{B}_0 - \lambda^2 \bm{I}_2 ) = (\lambda^2)^2 - \operatorname{tr}(\bm{B}_0) \lambda^2 + \det(\bm{B}_0),
\]
its discriminant is given by
\[
\Delta = \big(\operatorname{tr}(\bm{B}_0)\big)^2 - 4\det(\bm{B}_0) = \left(\frac{\varepsilon}{|D|}\right)^2 \left[ (c_{11}-c_{22})^2 + 4|c_{12}|^2 \right].
\]

\begin{itemize}
    \item For $\Delta \neq 0$, $\bm{B}_0$ has two distinct eigenvalues and is diagonalizable.
    
    \item For $\Delta = 0$,
    \[
    (c_{11}-c_{22})^2 + 4|c_{12}|^2 = 0 \implies c_{12} = 0 \quad \text{and} \quad c_{11} = c_{22},
    \]
    yielding
    \[
    \bm{B}_0 = \frac{\varepsilon}{|D|} c_{11} \bm{I}_2,
    \]
    which is diagonal (hence diagonalizable).
\end{itemize}

Thus $\bm{B}_0$ is always diagonalizable. However, Assumption~\ref{assump:system_f43} additionally requires $\det(\bm{B}_0) \neq 0$ to ensure the geometric multiplicity of the zero eigenvalue of $\bm{A}_0$ matches its algebraic multiplicity.

Assuming $c_{11}c_{22} - |c_{12}|^2 \neq 0$, we examine two mutually exclusive cases:
\begin{enumerate}[label=(\arabic*)]
    \item \textbf{$c_{12} = 0$}: Both $\bm{B}_0$ and $\bm{B}_1$ become diagonal matrices, precluding the formation of EPs.
    
    \item \textbf{$c_{12} \neq 0$}: The eigenvalues of $\bm{B}_0$ are
    \begin{align*}
        \lambda_1^2 &= \frac{\varepsilon }{2|D|} \left( c_{11} + c_{22} - \varsigma \right), \\
        \lambda_2^2 &= \frac{\varepsilon }{2|D|} \left( c_{11} + c_{22} + \varsigma \right),
    \end{align*}
    where $\varsigma = \sqrt{(c_{11} - c_{22})^2 + 4|c_{12}|^2} > 0$. There exists an invertible matrix $\bm{S}$ such that
    \begin{align*}
        \bm{S}^{-1}\bm{B}_0\bm{S} &= \begin{pmatrix} \lambda_1^2 & 0 \\ 0 & \lambda_2^2 \end{pmatrix}, \\
        \bm{S}^{-1}\bm{B}_1(t)\bm{S} &= \frac{\varepsilon }{2|D|\varsigma} \begin{pmatrix} b_{11} & b_{12} \\ b_{21} & b_{22} \end{pmatrix},
    \end{align*}
    with entries
    \begin{align*}
        b_{11}(t) &= \xi_{1}(t) c_{11}(c_{22}-c_{11}+\varsigma) + \xi_{2}(t), c_{22}(c_{11}-c_{22}+\varsigma) - 2 |c_{12}|^2 (\xi_{1}(t) + \xi_{2}(t)) \\
        b_{12}(t) &= -(\xi_1(t) - \xi_2(t))\left(2|c_{12}|^2 + c_{11}(c_{11} - c_{22} + \varsigma)\right), \\
        b_{21}(t) &= (\xi_1(t) - \xi_2(t))\left(2|c_{12}|^2 + c_{11}(c_{11} - c_{22} - \varsigma)\right), \\
        b_{22}(t) &= \xi_1(t) c_{11} (c_{11} - c_{22} + \varsigma) + \xi_2(t) c_{22}(c_{22} - c_{11} + \varsigma) + 2 |c_{12}|^2 (\xi_1(t) + \xi_2(t)).
    \end{align*}
    By Theorem \ref{th3}, EP formation requires
    \begin{align*}
        b_{12}^{(n_1-n_2)} = 0 \quad \text{or} \quad b_{21}^{(n_2-n_1)} = 0,
    \end{align*}
    which holds iff either
    \begin{align*}
        2|c_{12}|^2 + c_{11}(c_{11} - c_{22} + \varsigma) = 0
    \end{align*}
    or
    \begin{align*}
        2|c_{12}|^2 + c_{11}(c_{11} - c_{22} - \varsigma) = 0.
    \end{align*}
    Both equations imply $|c_{12}|^2 = c_{11}c_{22}$, equivalent to $c_{11}c_{22} - |c_{12}|^2 = 0$. This contradicts our initial assumption, making this case impossible.
\end{enumerate}
We conclude that Theorem \ref{th3} cannot construct EPs for system \eqref{f44} when $\det(\bm{B}_0) = c_{11}c_{22} - |c_{12}|^2 \neq 0$.


\end{document}